\documentclass[reqno]{amsart}

\usepackage{verbatim}
\usepackage{graphicx}
\usepackage{amssymb}
\usepackage{esint}

\newcommand{\R}{\mathbb{R}}

\newcommand{\grad}{\nabla}
\newcommand{\lap}{\Delta}

\newcommand{\normL}[2][\cdot]{\ensuremath{\|{#1}\|_{L^{#2}(\R^3)}}}

\theoremstyle{plain}
\newtheorem{thm}{Theorem}[section]
\newtheorem{lemma}[thm]{Lemma}
\newtheorem{defn}[thm]{Definition}
\newtheorem{cor}[thm]{Corollary}

\numberwithin{equation}{section}

\newtheorem{rem}[thm]{Remark}

\begin{document}
\title[Liquid Crystal Flow in $L^3_{\hbox{uloc}}(\R^3)$]{Well-Posedness of Nematic Liquid Crystal Flow in $L^3_{\hbox{uloc}}(\R^3)$} 

\author[J. Hineman]{Jay Lawrence Hineman\ \ \ } 
\author[C. Wang]{\ \ \ Changyou Wang} 
\address{Department of Mathematics\\
University of Kentucky \\
Lexington, KY 40506}

\email{hineman@ms.uky.edu, \ cywang@ms.uky.edu} 
\date{\today}
\subjclass[2000]{}

\begin{abstract}
In this paper, we establish the local well-posedness for the Cauchy problem of the simplified version of hydrodynamic flow of nematic
liquid crystals (\ref{LLF}) in $\mathbb R^3$ for any initial data $(u_0,d_0)$ 
having small $L^3_{\hbox{uloc}}$-norm of $(u_0,\nabla d_0)$. Here $L^3_{\hbox{uloc}}(\mathbb R^3)$ is
the space of uniformly locally $L^3$-integrable functions. For any initial data $(u_0, d_0)$ with small $\displaystyle
\|(u_0,\nabla d_0)\|_{L^3(\mathbb R^3)}$, we show that there exists a unique,
global solution to (\ref{LLF}) which is smooth for $t>0$ and has monotone deceasing $L^3$-energy  for $t\ge 0$.
\end{abstract}
\maketitle 

\section{Introduction}
\setcounter{equation}{0}
\setcounter{thm}{0}
In this paper, we consider the Cauchy problem for the following hydrodynamic system modeling the flow of nematic liquid crystal materials in $\mathbb R^3$:
for $0<T\le\infty$ and $(u,P,d):\R^3\times[0,T) \to \R^3\times\R\times S^2$, the system is given by
\begin{equation}\label{LLF}
\begin{cases}
\begin{aligned}
u_t + u \cdot \grad u - \nu\lap u + \grad P &= -\lambda\grad\cdot(\grad d \odot \grad d), \ \ {\rm{in}}\ \mathbb R^3\times (0,T),\\
\grad \cdot u &= 0, \ \ \ \ \ \ \ \ \ \  \ \ \ \ \ \ \ \ \ \ \ \ \ {\rm{in}}\ \mathbb R^3\times (0,T),\\
d_t + u \cdot \grad d &=\gamma(\lap d+ |\grad d|^2 d), \ \ \ \ {\rm{in}}\ \mathbb R^3\times (0,T),\\
(u, d)&=(u_0, d_0), \qquad\qquad \ \ \ {\rm{on}}\ \mathbb R^3\times \{0\},
\end{aligned}
\end{cases}
\end{equation}
for a given initial data $(u_0,d_0):\mathbb R^3\to\mathbb R^3\times S^2$ with $\nabla\cdot u_0=0$.
Here $u:\mathbb R^3\to\mathbb R^3$ represents the velocity field of the fluid, $d:\mathbb R^3\to S^2$ -- the unit sphere
in $\mathbb R^3$ -- is a unit vector field representing the macroscopic molecular orientation of the nematic liquid crystal
material, $P:\mathbb R^3\to\R$ represents the pressure function. The constants $\nu,\lambda,$ and $\gamma$ are positive
constants that represent the viscosity of the fluid, the competition between kinetic and potential energy, and the microscopic
elastic relaxation time for the molecular orientation field.
$\nabla\cdot$ denotes the divergence operator in $\mathbb R^3$, and $ \grad d \odot \grad d$ denotes the
symmetric $3\times 3$ matrix:
\begin{equation*}
\label{def:odot}
\left ( \grad d \odot \grad d \right )_{ij} = \langle\grad_id,  \grad_j d\rangle, \ 1\le i, j \le 3.
\end{equation*}
Throughout this paper, we denote $\displaystyle\langle v, w\rangle$ or $v\cdot w$
as the inner product in $\mathbb R^3$ for $v,w\in\mathbb R^3$.

The system (\ref{LLF}) is a simplified version of the famous Ericksen-Leslie model for the hydrodynamics of nematic liquid crystals developed by Ericksen
and Leslie during the period of 1958 through 1968 \cite{ericksen, leslie, degenes}. This system reduces to the Ossen-Frank
model in the static theory of liquid crystals. It is a macroscopic continuum description of the time evolution of the materials
under the influence of flow field $u$ and the macroscopic description of the microscopic orientation field $d$ of
rod-like liquid crystals. The current form of system (\ref{LLF}) was first proposed by Lin \cite{lin} back in the late 1980's. From the mathematical point
of view, (\ref{LLF}) is a system coupling the non-homogeneous incompressible Navier-Stokes equation and the
transported heat flow of harmonic maps to $S^2$. Lin-Liu \cite{LL1, LL2} initiated the mathematical analysis of (\ref{LLF}) by
considering its Ginzburg-Landau approximation or the so-called orientation with variable degrees in the terminology of Ericksen.
Namely, the Dirichlet energy $\displaystyle\int \frac12|\nabla d|^2$ for $d:\mathbb R^3\to S^2$ is replaced by  the Ginzburg-Landau
energy $\displaystyle\int \frac12|\nabla d|^2+\frac{1}{4\epsilon^2}(1-|d|^2)^2$ ($\epsilon>0$) for $d:\mathbb R^3\to\mathbb R^3$.
Hence (\ref{LLF})$_3$ is replaced by
\begin{equation}\label{GL_LLF}
\partial _t d+u\cdot\nabla d= \gamma(\Delta d+\frac{1}{\epsilon^2} (1-|d|^2)d).
\end{equation}
Lin-Liu proved in \cite{LL1, LL2} (i) the existence of a unique, global smooth solution in dimension two and in dimension three under large
viscosity $\nu$; and (ii) the existence of suitable weak solutions and their partial regularity in dimension three, analogous to the
celebrated regularity theorem by Caffarelli-Kohn-Nirenberg \cite{CKN} for the three-dimensional incompressible Navier-Stokes equation.

As already pointed out by \cite{LL1, LL2}, it is a very challenging problem to study the convergence of solutions $(u_\epsilon, P_\epsilon, d_\epsilon)$ to (\ref{LLF})$_1$-(\ref{LLF})$_2$-(\ref{GL_LLF}) when $\epsilon\downarrow 0$. In particular, the existence of global Leray-Hopf type  weak solutions to the initial and boundary value problem of (\ref{LLF}) has only been established recently by Lin-Lin-Wang \cite{LLW} in dimension two, see also Hong \cite{hong} and Xu-Zhang \cite{XZ} and Hong-Xin \cite{HX} for related works.

Because of the super-critical nonlinear term $\nabla\cdot(\nabla d\odot\nabla)$ in (\ref{LLF})$_1$,
it has been an outstanding open problem whether there exists a global  Leray-Hopf type weak solution to (\ref{LLF}) in $\mathbb R^3$ for any initial data $(u_0, d_0)\in L^2(\R^3,\R^3)\times \dot{W}^{1,2}(\R^3, S^2)$ with $\nabla\cdot u_0=0$.
It is standard that in $\mathbb R^3$ the local existence of a unique, strong solution to (\ref{LLF}) can be obtained for any initial data
$u_0\in W^{s,2}(\mathbb R^3)$ and $d_0\in W^{s+1,2}(\mathbb R^3, S^2)$ for $s>3$ with $\nabla\cdot u_0=0$, see for example \cite{wen-ding}.
A blow-up criterion for local strong solutions to
(\ref{LLF}), similar to the Beale-Kato-Majda criterion for the Navier-Stokes equation (see \cite{BKM}), was obtained by Huang-Wang
\cite{HW1}. For small initial data in certain Besov spaces, Li-Wang \cite{li-wang} obtained the global existence of strong solutions
to (\ref{LLF}).  We would like to mention that  Wang \cite{wang} has recently obtained the global (or local) well-posedness of
(\ref{LLF}) for initial data $(u_0, d_0)$ belonging to possibly the largest space ${\rm{BMO}}^{-1}\times {\rm{BMO}}$ with $\nabla\cdot u_0=0$, which is a invariant space under parabolic scaling associated with (\ref{LLF}),  with small norms.

In this paper, we are mainly interested in the local well-posedness of (\ref{LLF}) for any initial data $(u_0, d_0)$ such that
$(u_0,\nabla d_0)\in L^3_{\rm{uloc}}(\R^3)$. Henceforth $L^3_{\hbox{uloc}}(\R^3)$  denotes
the space of uniformly locally $L^3$-integrable functions.  It turns out that $L^3_{\hbox{uloc}}(\mathbb R^3)$ is also
invariant under parabolic scaling associated with
(\ref{LLF}).

Now we give the definition of $L^3_{\hbox{uloc}}(\mathbb R^3)$.  The readers can consult
the monograph by Lemari\'e-Rieusset \cite{LR}  for applications of the space $L^3_{\hbox{uloc}}(\mathbb R^3)$
to the Navier-Stokes equation.

\begin{defn}\label{l3space} A function $f\in L^3_{\rm{loc}}(\mathbb R^3)$ belongs to the space
$L^3_{\rm{uloc}}(\mathbb R^3)$ consisting of uniformly locally $L^3$-integrable functions, if there exists $0<R<+\infty$ such that
\begin{equation}\label{l3space_norm}
\left\|f\right\|_{L^3_R(\mathbb R^3)}
:=\sup_{x\in\mathbb R^3}\Big (\int_{B_R(x)}|f|^3\Big)^\frac13<+\infty.
\end{equation}
\end{defn}

\smallskip
It is clear that
\begin{itemize}
\item $L^3(\mathbb R^3)\subset L^3_{\hbox{uloc}}(\mathbb R^3)$.\\
\item If $f\in L^3_{\hbox{uloc}}(\mathbb R^3)$,
then $\displaystyle \|f\|_{L^3_R(\mathbb R^3)}$ is finite for any $0<R<+\infty$. For any two $0<R_1\le R_2<\infty$,
it holds
\begin{equation}\label{equiv_norm}
\left\|f\right\|_{L^3_{R_1}(\mathbb R^3)}
\le \left\|f\right\|_{L^3_{R_2}(\mathbb R^3)} \lesssim \left(\frac{R_2}{R_1}\right)\left\|f\right\|_{L^3_{R_1}(\mathbb R^3)},
\ \forall \ f\in L^3_{\hbox{uloc}}(\mathbb R^3).
\end{equation}
\item $\displaystyle L^3_{\hbox{uloc}}(\mathbb R^3)\subset \bigcap_{0<R<\infty} {\rm{BMO}}^{-1}_R(\mathbb R^3)$ (see \cite{koch-tataru} or \cite{wang}). Moreover, for any
$0<R<\infty$, it holds
\begin{equation}\label{inclusion}
\left[f\right]_{\rm{BMO}^{-1}_R(\mathbb R^3)}\lesssim \left\|f\right\|_{L^3_{R}(\mathbb R^3)},
\ \forall \ f\in L^3_{\hbox{uloc}}(\mathbb R^3).
\end{equation}
\end{itemize}

\medskip
Throughout this paper, we write $A\lesssim B$ if there exists a universal constant $C>0$ such that $A\le CB$.
Here are a few more notations and conventions that we will use through this paper. For two matrices $M, N$ of order $3$, we use
$\displaystyle M:N=\sum_{1\le i, j\le 3} M^{ij} N^{ij}$ to denote their scalar product.  For two vectors $u, v\in\mathbb R^3$,
we let $u\otimes v$ denote  their tensor product: $\displaystyle (u\otimes v)_{ij}=u^i v^j, \ 1\le i,  j\le 3$.
For $0<s<+\infty$ and $1\le p\le\infty$, we denote by $W^{s,p}(\mathbb R^3)$ and $\dot{W}^{s,p}(\mathbb R^3$
as the Sobolev space and the homogeneous Sobolev spaces respectively.
For $0\le a<b<+\infty$, denote
$$C_b^\infty(\mathbb R^3\times [a,b])=\bigcap_{m\ge 0}\Big\{
f\in C^m(\mathbb R^3\times [a,b]) \ :\ \|f\|_{C^m(\R^3\times [a,b])}<+\infty\Big\},$$
$$L^\infty([a,b], L^3_{\hbox{uloc}}(\R^3))
=\Big\{f\in L^\infty([a,b], L^3_{1}(\R^3))\Big\},$$
and 
$$C_*^0([a,b], L^3_{\rm{uloc}}(\R^3))
=\Big\{f\in C((a,b], L^3_1(\R^3))\cap L^\infty([a,b], L^3_1(\R^3)): \
 \hbox{as}\ t\downarrow 0,
f(t) \rightarrow f(a) \hbox{ in } L^3_{\rm{loc}}(\R^3)\Big\}.
$$
Repeated indices are summed unless specificized otherwise. Upper indices denote components and lower indices denote derivatives.

\bigskip

Now we state our main theorem.

\begin{thm}
  \label{wellposedness}
  There exist $\epsilon_0 > 0$ and $\tau_0 > 0$ such that if $u_0:\R^3 \to \R^3$, with $\nabla\cdot u_0=0$, and $d_0:\R^3 \to S^2$ satisfies
$(d_0-e_0)\in L^3(\R^3)$ for some $e_0\in S^2$, and
  \begin{equation}\label{small_condition}
   ||| (u_0, \nabla d_0)|||_{L^3_R(R^3)}:
=\sup_{x\in\R^3} \left(\int_{B_R(x)} |u_0|^3 +|\nabla d_0|^3\right)^\frac13 \leq \epsilon_0
  \end{equation}
for some $0<R<\infty$,
  then there exist $T_0\ge \tau_0 R^2$ and a unique solution $(u, d):\R^3 \times [0,T_0) \to \R^3 \times\R\times S^2$ of (\ref{LLF}) such that
the following properties hold:\\
(i) For $t\downarrow 0$, 
$\displaystyle (u(t), d(t))\rightarrow (u_0, d_0)$ and $\nabla d(t)\rightarrow \nabla d_0$
in $L^3_{\rm{loc}}(\R^3)$. \\
(ii) $$(u,d) \in \bigcap_{0<\delta<T_0} C^\infty_b(\R^3\times [\delta,T_0-\delta], \mathbb R^3\times S^2), \
\ (u,\nabla d)\in\bigcap_{0<T'<T_0} C_*^0([0, T'], L^3_{\rm{uloc}}(\mathbb R^3)).$$
(iii) \begin{equation}\label{small_condition0}
|||(u(t),\nabla d(t))|||_{L^\infty([0,\tau_0 R^2], L^3_R(\R^3))}\le C\epsilon_0.
\end{equation}
(iv) If $T_0 < +\infty$ is the maximum time interval then it must hold
  \begin{equation}\label{blowup_criterion}
    \limsup_{t \uparrow T_0} |||(u(t), \nabla d(t)) |||_{L_{r}^3(\R^3)}> \epsilon_0,
\ \forall \ 0<r<\infty.
  \end{equation}
\end{thm}

The ideas to prove Theorem \ref{wellposedness} are motivated by those employed by \cite{LLW}.
There are five main ingredients, which include
\begin{itemize}
\item approximate  $(u_0,d_0)$ by smooth $(u_0^k, d_0^k)$ (see Lemma \ref{approx} below)
and obtain $0<T_k<+\infty$ and a sequence of smooth solutions $(u^k, P^k, d^k)$
of (\ref{LLF}) in $\mathbb R^3\times [0, T_k]$, under the initial data $(u_0^k, d_0^k)$;
\item utilizing the local $L^3$-energy inequality (\ref{localEnergyInequality0}), obtain uniform lower bounds of $T_k$;
\item apply the $\epsilon_0$-regularity Theorem \ref{smooth:solutions} to 
obtain a priori  derivative estimates of $(u^k, d^k)$ and then take limit to obtain the local existence of
$L^3_{\rm{uloc}}$-solutions
to (\ref{LLF}); 
\item apply Theorem \ref{smooth:solutions} again to characterize the finite maximal time
interval; and
\item adapt the proof of \cite{wang} to show the uniqueness.
\end{itemize}

\medskip
For a solution $(u, P, d)$ to (\ref{LLF}),  denote
its $L^3$-energy by
$$E_3(u,\nabla d)(t)=\int_{\mathbb R^3}(|u(t)|^3+|\nabla d(t)|^3), \ \ t\ge 0.$$
Concerning the global well-posedness of (\ref{LLF}), we have 

\begin{thm}\label{gwp}
  There exists an $\epsilon_0 > 0$ such that if $(u_0,d_0) \in L^3(\R^3,\R^3) \times \dot{W}^{1,3}(\R^3,S^2)$,
with $\nabla\cdot u_0=0$,  satisfies
  \begin{equation}\label{small_l3}
    E_3(u_0,\nabla d_0)  \leq \epsilon_0^3,
  \end{equation}
  then there exists a unique global solution $(u,d):\mathbb R^3\times [0,\infty)
\to \mathbb R^3\times\mathbb R\times S^2$  of (\ref{LLF}) such that
$\displaystyle (u, d) \in C^\infty(\mathbb R^3\times (0,+\infty))\cap C([0,\infty), L^3(\R^3)\times \dot{W}^{1,3}(\R^3))$,
$E_3(u,\nabla d)(t)$ is monotone decreasing for $t\ge 0$, and
\begin{equation}\label{uniform_gradient_estimate}
\left\|\nabla^m u(t)\right\|_{L^\infty(\R^3)}+\left\|\nabla^{m+1} d(t)\right\|_{L^\infty(\R^3)}\le \frac{C\epsilon_0}{t^{\frac{m}2}},
\ \ \forall\ t>0, \ m\ge 0.
\end{equation}
\end{thm}

We mention here that the first conclusion of Theorem \ref{gwp} has been  proven by  \cite{ding-lin}, which is based
on  refinement of the argument by Wang \cite{wang}. 
Since the exact values of $\nu, \lambda, \gamma$ don't play a role in this paper, we
henceforth assume
$$\nu=\lambda=\gamma=1.$$

\medskip
The paper is written as follows. In \S2, we derive an inequality for the global $L^3$-energy of smooth solutions of
(\ref{LLF}). In \S3, we derive an inequality for the local $L^3$-energy of smooth solutions of (\ref{LLF}) and prove Theorem \ref{gwp}.
In \S4, we will prove an $\epsilon_0$-regularity for suitable weak solutions to (\ref{LLF}). In particular, 
 a priori derivative estimates hold for smooth solutions to (\ref{LLF}) under a smallness condition.
In \S5, we will prove Theorem \ref{wellposedness}.

\medskip

\section{Inequality on the global $L^3$-energy and proof of Theorem \ref{gwp}}

In this section, we will derive an inequality for the $L^3$-energy $E_3(u,\nabla d)(t)$
for any smooth solution $(u, d):\mathbb R^3\times [0, T]\to \mathbb R^3\times \mathbb S^2$, for
$0<T\le\infty$,  of the system (\ref{LLF}) for nematic liquid crystals.

\begin{lemma}\label{L3-inequality}
There exists $C>0$ such that for $0<T\le\infty$ if $(u,d)\in C^\infty(\mathbb R^3\times [0, T), \mathbb R^3\times S^2)\cap C([0,T),
L^3(\mathbb R^3)\times \dot{W}^{1,3}(\mathbb R^3))$ and $P\in L^\infty([0,T), L^\frac32(\R^3))$ solves (\ref{LLF}), then
it holds
\begin{equation}\label{L3-inequality1}
  \begin{aligned}
    &\frac{d}{dt} \int_{\R^3} (|u|^3 + |\grad d |^3)
    + \left [1 - C \normL[u]{3}^2 \right ] \int_{\R^3} |u||\grad u|^2\\
    &+ \left [1 - C(\normL[u]{3} + \normL[u]{3} \normL[\grad d]{3} + \normL[\grad d]{3}^2) \right ] \int_{\R^3} |\grad d||\grad^2 d|^2 \leq 0.
  \end{aligned}
\end{equation}
\end{lemma}

\begin{proof}
Taking spatial derivatives of (\ref{LLF})$_3$, multiplying the resulting equation by $|\nabla d|\nabla d$, and integrating over $\mathbb R^3$, we have
\begin{equation}
  \label{int:dir}
 \frac{d}{dt} \int_{\R^3}\frac13 |\grad d|^3
  = {\int_{\R^3}  \grad (\lap d):|\grad d|\grad d}
  - \int_{\R^3}  \grad (u\cdot\grad d):|\grad d|\grad d
  - \int_{\R^3}  \grad (|\grad d|^2d):|\grad d|\grad d.
\end{equation}
For terms on the right hand side of (\ref{int:dir},  by integration by parts we have
\begin{equation*}
  \label{int:dir:est:a}
  \begin{aligned}
    \int_{\R^3}  \grad (\lap d):|\grad d|\grad d
    &= -\int_{\R^3\cap\{|\nabla d|>0\}}  \nabla^2 d: \nabla (|\grad d|\nabla d) \\
    &= -\int_{\R^3\cap\{|\nabla d|>0\}}  (|\grad d||\grad^2 d|^2 + \frac{|\grad^2d \cdot \grad d|^2}{|\grad d|} )\\
&\le  -\int_{\R^3} |\grad d| |\grad^2 d|^2,
  \end{aligned}
\end{equation*}
\begin{equation*}
  \label{int:dir:est:b}
  \begin{aligned}
    \int_{\R^3}  \grad (u\cdot\grad d):|\grad d|\grad d
   & = - \int_{\R^3} (u\cdot \grad d) \cdot ((\grad |\grad d|)\cdot \grad d +  |\grad d|\lap d)\\
&\lesssim \int_{\mathbb R^3} |u||\nabla d|^2|\nabla^2 d|,
  \end{aligned}
\end{equation*}
and, using $|d|=1$,
\begin{equation*}
  \label{int:dir:est:c}
  \begin{aligned}
    \int_{\R^3} \grad (|\grad d|^2 d):|\grad d| \grad d
    &= \int_{\R^3} (\grad |\grad d|^2)\cdot |\grad d| \grad (\frac{|d|^2}2) + \int_{\R^3} |\grad d|^2 \grad d:|\grad d|^2 \grad d\\
&= \int_{\R^3} |\grad d|^5.
  \end{aligned}
\end{equation*}
Putting these estimates into \eqref{int:dir}
yields
\begin{equation}
  \label{dir:diff:ineq:1}
  \frac{d}{dt}\int_{\R^3} |\grad d|^3 + \int_{\R^3} |\grad(|\grad d|^\frac32)|^2
  \lesssim \int_{\R^3} |\grad d |^5 + |u||\grad d|^2|\grad^2 d|,
\end{equation}
where have used the following variant of the Kato inequality
\begin{equation*}
|\grad |\grad d|^{\frac32}| = \frac{3}{2} |\grad d |^{\frac12} |\grad | \grad d||
\le  \frac{3}{2} |\grad d |^{\frac12} |\grad^2 d|.
\end{equation*}
Observe that  by the Sobolev inequality and the Kato inequality above, we have
\begin{equation}
  \begin{aligned}
    \int_{\R^3} |\grad d |^9 = \int_{\R^3} (|\grad d |^{\frac32})^6
    \lesssim \left (\int_{\R^3} |\grad | \grad d |^{\frac32}|^2 \right )^3
    \lesssim \left (\int_{\R^3} |\grad d| |\grad^2 d|^2 \right )^3 .
  \end{aligned}
  \label{embedEst1}
\end{equation}
Hence, by the H\"older inequality and (\ref{embedEst1}), we have
\begin{equation*}
\|\grad d \|_{L^5(\mathbb R^3)}^5
\leq \|\grad d\|_{L^3(\mathbb R^3)}^2\|\nabla d\|_{L^9(\mathbb R^3)}^3
\lesssim  \left (\int_{\R^3} |\grad d |^3 \right)^{2/3}\left(\int_{\R^3} |\grad d | |\grad^2 d|^2 \right).
\end{equation*}
For the second term on the right-hand side of (\ref{dir:diff:ineq:1}), by the
H\"older inequality and (\ref{embedEst1}) we have
\begin{equation*}
  \begin{aligned}
    \int_{\R^3} |u||\grad d|^2|\grad^2 d|
    &\leq \normL[u]{3} \normL[|\grad d|^{\frac32}]{6} \normL[|\grad d|^{\frac12}|\grad^2 d|]{2} \\
    &\lesssim \normL[u]{3} \normL[\grad|\grad d|^{\frac32}]{2} \normL[|\grad d|^{\frac12}|\grad^2 d|]{2} \\
    &\lesssim \normL[u]{3} \normL[|\grad d|^{\frac12}|\grad^2 d|]{2}^2.
  \end{aligned}
\end{equation*}
Inserting these two estimates into \eqref{dir:diff:ineq:1} yields
\begin{equation}
 \frac{d}{dt} \int_{\R^3}|\grad d|^3 + \left [ 1 - C \left(\normL[\grad d]{3}^2+ \normL[u]{3} \right ) \right ] \int_{\mathbb R^3} |\grad d||\grad^2 d|^2 \leq 0.\label{dir:diff:ineq:3}
\end{equation}

Next we estimate the $L^3$-norm of $u$.
Multiplying (\ref{LLF})$_1$  by $|u|u$ and integrating over $\R^3$ gives
\begin{equation}
  \label{NS:int}
  \begin{aligned}
    &\frac{d}{dt}\int_{\R^3} \frac13 |u|^3 \\
    &=
    \int_{\R^3} \lap u\cdot |u|u
    -\int_{\R^3} (u \cdot \grad u) \cdot |u|u
    -\int_{\R^3} \grad P \cdot |u|u
    - \int_{\R^3} \left ( \grad\cdot(\grad d \odot \grad d) \right )\cdot |u|u.
  \end{aligned}
\end{equation}
For the terms on the right hand side of (\ref{NS:int}),  by integration by parts we have
\begin{equation*}
  \begin{aligned}
    \int_{\R^3} (\lap u)\cdot |u|u
    = - \int_{\R^3} |\grad u|^2 |u| + |u||\grad|u||^2,
  \end{aligned}
\end{equation*}
\begin{equation*}
  \begin{aligned}
    \int_{\R^3} (u\cdot \grad u)\cdot |u|u
    = \int_{\R^3} u\cdot \grad\left (\frac{|u|^3}{3} \right)
    = 0,
  \end{aligned}
\end{equation*}
\begin{equation*}
  \begin{aligned}
    \int_{\R^3} \grad P \cdot |u| u
    = - \int_{\R^3} P u\cdot \grad |u| + P|u|(\grad \cdot u)
    = - \int_{\R^3} P u\cdot \grad |u|,
  \end{aligned}
\end{equation*}
and
\begin{equation*}
  \begin{aligned}
    - \int_{\R^3} (\grad \cdot (\grad d \odot \grad d))\cdot |u| u
    &= \int_{\R^3} (\grad d \odot \grad d):\grad(|u|u)  \\
    &= \int_{\R^3} (\grad d \odot \grad d):\grad |u|\otimes u +|u|
 (\grad d \odot \grad d) : \grad u  \\
    &\lesssim\int_{\R^3} |\grad d|^2 |u| |\grad u|.
  \end{aligned}
\end{equation*}
Substituting these estimates into \eqref{NS:int}, we obtain
\begin{equation}
  \begin{aligned}
    \frac{d}{dt} \int_{\R^3} |u|^3 + \int_{\R^3} |u||\grad u|^2 |u|
    \lesssim \int_{\R^3} |P| |u||\grad |u|| + \int_{\R^3} |\grad d|^2 |u| |\grad u|.
  \end{aligned}
  \label{NS:diff:ineq:1}
\end{equation}
Using the  Kato inequality $|\grad |u|| \leq |\grad u|$, the Cauchy inequality and the H\"older inequality in \eqref{NS:diff:ineq:1}, we obtain
\begin{equation*}
  \begin{aligned}
   \frac{d}{dt} \int_{\R^3} |u|^3 + \int_{\R^3} |u||\grad u|^2
    &\leq C\int_{\R^3}|u| (|P|^2 + |\grad d|^4)+ \frac{1}{2}\int_{\R^3} |u||\grad u|^2 \\
    &\leq C (\normL[P]{3}^2 +  \normL[\grad d]{6}^4)\normL[u]{3} + \frac{1}{2}\int_{\R^3} |u||\grad u|^2.
  \end{aligned}
\end{equation*}
Therefore we get
\begin{equation}
  \label{NS:diff:ineq:2}
  \frac{d}{dt} \int_{\R^3} |u|^3 + \int_{\R^3}|u| |\grad u|^2 \lesssim (\normL[P]{3}^2 +  \normL[\grad d]{6}^4)\normL[u]{3}.
\end{equation}
We need to estimate $\|P\|_{L^3(\R^3)}$.
To do so, we take divergence of \eqref{LLF}$_1$ to obtain
\begin{equation}
  \begin{aligned}
    -\lap P
    &= \grad \cdot \grad \cdot (u \otimes u + \grad d \odot \grad d). \\
  \end{aligned}
\label{PressPoisson}
\end{equation}
Set
\begin{equation*}
   g^{jk} :=  u^j u^k + \grad_j d \cdot \grad_k d, \ 1\le j, k\le 3.
\end{equation*}
Then we have
\begin{equation}
  P=\Delta^{-1}(\nabla_{jk}^2 g^{jk})=- {\mathbf R}_j{\mathbf R}_k(g^{jk}).
  \label{rieszPress}
\end{equation}
Henceforth $\displaystyle {\mathbf R}_j=(-\Delta)^{-\frac12}\nabla_j$ denotes the $j^{th}$-Riesz transform on $\mathbb R^3$
for $1\le j\le 3$.

Since ${\mathbf R}_j:L^q(\R^3) \to L^q(\R^3)$ is bounded for $1 < q < \infty$ (see Stein \cite{stein}),
 we have
\begin{equation}
  \normL[P]{3} = \normL[{\mathbf R}_j{\mathbf R}_k(g^{jk})]{3} \lesssim \normL[g^{jk}]{3} \lesssim \normL[u]{6}^2 + \normL[\grad d]{6}^2.
  \label{pressBound}
\end{equation}
Inserting \eqref{pressBound} into \eqref{NS:diff:ineq:2} yields
\begin{equation}
  \label{NS:diff:ineq:4}
  \frac{d}{dt} \int_{\R^3} |u|^3 + \int_{\R^3}|u| |\grad u|^2
\lesssim (\normL[u]{6}^4 +  \normL[\grad d]{6}^4)\normL[u]{3}.
\end{equation}
Using the H\"older inequality, the Sobolev inequality, and
$|\grad |u|^{\frac32}| \lesssim |\grad u||u|^{\frac12}$,  we have
\begin{equation*}
  \label{uSqEstimate}
  \begin{aligned}
    \normL[u]{6}^4
    \leq \normL[u]{3}\normL[u]{9}^3
    \lesssim \normL[u]{3} \normL[\grad |u|^{\frac32}]{2}^2
    \leq \normL[u]{3} \int_{\R^3} |u||\grad |u||^2.
  \end{aligned}
\end{equation*}
Similarly we have
\begin{equation*}
  \label{graddSqEstimate}
  \begin{aligned}
    \normL[\grad d]{6}^4
    \lesssim \normL[\grad d]{3} \int_{\R^3} |\grad d||\grad^2 d|^{2}.
  \end{aligned}
\end{equation*}
Substituting these two estimates into \eqref{NS:diff:ineq:4}, we obtain
\begin{equation}
  \label{NS:diff:ineq:5}
  \begin{aligned}
    \frac{d}{dt} \int_{\R^3} |u|^3 + \int_{\R^3} |u|  |\grad u|^2
    \lesssim \normL[u]{3}^2 \int_{\R^3}  |u| |\grad u|^2
 + \normL[u]{3} \normL[\grad d]{3} \int_{\R^3} |\grad d| |\grad^2 d|^2.
  \end{aligned}
\end{equation}
Combining \eqref{dir:diff:ineq:3} and \eqref{NS:diff:ineq:5} yields (\ref{L3-inequality1}). \end{proof}

\begin{cor}\label{mono_l3}
There exists $\epsilon_0 > 0$ such that for $0<T\le\infty$,
if $(u,d)\in C^\infty(\mathbb R^3\times [0,T), \mathbb R^3\times S^2)\cap L^\infty([0,T), L^3(\R^3)\times \dot{W}^{1,3}(\R^3))$
is a solution to (\ref{LLF}) satisfying
\begin{equation}\label{small3energy}
E_3(u_0,\nabla d_0)\leq \epsilon_0^3,
\end{equation}
then $E_3(u(t), \nabla d(t))$ is monotone decreasing for $0\le  t<T$.
\end{cor}
\begin{proof}
 Denote
 \begin{equation*}
  E_3(t):=E_3(u(t), \nabla d(t)),\ \  t\ge 0.
  \end{equation*}
Let $T_{\rm{max}}\in [0, T)$ be defined by
\begin{equation*}
T_{\rm{max}}=\max\Big\{t\in [0, T): \ E_3(s)\le 2\epsilon_0^3,  \ \forall\  0\le s\le t\Big\}.
\end{equation*}
By continuity and  (\ref{small3energy}), we have that $0<T_{\rm{max}}\le T$, and
  \begin{equation}
    \label{gen}
    E_3(t) \leq 2\epsilon_0^3, \ 0\le t<T_{\rm{max}}, \ E_3(T_{\rm{max}})=2\epsilon_0^3.
  \end{equation}
Suppose $T_{\rm{max}}<T$. Choose $\epsilon_0 > 0$ so small that
  \begin{equation*}
    1 - C\epsilon_0^2 \geq 0 \text{ and } 1 - C(\epsilon_0 + 2\epsilon_0^2) \geq 0.
  \end{equation*}
Then (\ref{gen}) and (\ref{L3-inequality1}) imply that
  \begin{equation*}
    \frac{d}{dt} E_3(t)\le \frac{d}{dt} E_3(t)
    + \left [1 - C \epsilon_0^2 \right ] \int_{\R^3} u||\grad u|^2
    + \left [1 - C(\epsilon_0 + 2\epsilon_0^2) \right ]
    \int_{\R^3} |\grad d||\grad^2 d|^2
    \leq 0
  \end{equation*}
  holds for $0\le t\le T_{\rm{max}}$. Hence  $E_3(t)$ is decreasing in $[0, T_{\rm{max}}]$ and 
  \begin{equation*}
    E_3(T_{\rm{max}}) \leq E_3(0) \leq \epsilon_0^3 < 2\epsilon_0^3.
  \end{equation*}
This contradicts the definition of $T_{\rm{max}}$. Thus $T_{\rm{max}} = T$ and $E_3(t)$ is monotone decreasing
in $[0,T)$.  \end{proof}

\bigskip
\noindent{\bf Proof of Theorem \ref{gwp}}: Since $C^\infty(\R^3, S^2)$ is dense in $\dot{W}^{1,3}(\mathbb R^3,S^2)$
(see \cite{SU}), it is not hard to show that there exist
$\{(u_0^k,d_0^k)\}\subset C^\infty(\R^3,\R^3)\times C^\infty(\R^3, S^2)$ such that
$$\nabla\cdot u_0^k=0\ {\rm{in}}\ \R^3,
\ \ \ \lim_{k\rightarrow \infty}(\|u_0^k-u_0\|_{L^3(\R^3)}+\|\nabla(d_0^k-d_0)\|_{L^3(\R^3)})=0.$$
Consider the system (\ref{LLF}) under the initial condition $\displaystyle (u,d)|_{t=0}=(u_0^k, d_0^k)$. It is standard that
there exist $T_k>0$ and smooth solutions $\displaystyle(u_k,d_k)\in C^\infty(\R^3\times [0,T_k], \R^3\times S^2)\cap
C([0,T_k], L^3(\R^3)\times \dot{W}^{1,3}(\R^3))$ to (\ref{LLF}).

Since $E_3(u_0,\nabla d_0)\le\epsilon_0^3$, we may assume that $E_3(u_0^k, \nabla d_0^k)\le 2\epsilon_0^3$ for all $ k\ge 1.$
Hence by Corollary \ref{mono_l3}, we conclude that
$$\sup_{0\le t\le T_k} E_3(u^k(t), \nabla d^k(t))\le E_3(u_0^k, \nabla d_0^k)\le 2\epsilon_0^3, \ \forall\ k\ge 1.$$
For the corresponding pressure functions $P^k$, since
$$\Delta P^k=-\nabla\cdot\nabla\cdot(u^k\otimes u^k+\nabla d^k\odot\nabla d^k) \ {\rm{in}}\ \R^3,$$
we have
$$\sup_{0\le t\le T_k}\|P^k\|_{L^\frac32(\R^3)}\lesssim \sup_{0\le t\le T_k}
(\|u^k\|_{L^3(\R^3)}^2+\|\nabla d^k\|_{L^3(\R^3)}^2)\le C\epsilon_0^2.$$
Let $T_k$ be the maximal time interval for $(u_k, d_k)$. If $0<T_k<+\infty$, then by Theorem \ref{smooth:solutions} in \S 4 below
we conclude that $(u_k, d_k)\in C^\infty_b(\R^3\times [0, T_k], \R^3\times S^2)$. Hence
$\displaystyle (u_k(T_k), d_k(T_k))\in C^\infty(\R^3, \R^3\times S^2)\cap L^3(\R^3)\times\dot{W}^{1,3}(\R^3)$,
and
$$E_3((u_k(T_k), \nabla d_k(T_k)))\le 2\epsilon_0^3$$
so that we can extend the smooth solutions $(u_k, d_k)$ beyond the time $T_k$, which would contradict the maximality of $T_k$.
Therefore $T_k=\infty$ and the smooth solution $(u_k,d_k)$ exists globally. Moreover,
$E_3(u_k(t),\nabla d_k(t))$ is monotone decreasing and less than $2\epsilon_0^3$. By Theorem \ref{smooth:solutions}, we
have the derivative estimates:
\begin{equation}\label{uniform_grad_est}
\left\|\nabla^m u_k(t)\right\|_{L^\infty(\R^3)}+\left\|\nabla^{m+1} d_k(t)\right\|_{L^\infty(\R^3)}\le \frac{C\epsilon_0}{t^{\frac{m}2}},
\ \forall \ t>0, \ m\ge 1.
\end{equation}
After taking possible subsequences, we may assume that there exists $\displaystyle (u, d)\in C^\infty(\R^3\times (0,+\infty), \R^3\times S^2)
\cap C([0,+\infty), L^3(\R^3)\cap \dot{W}^{1,3}(\R^3))$ such that as $k\rightarrow\infty$, \\
(1) $(u_k,d_k)\rightarrow (u, d)$ in $C^m_{\rm{loc}}(\R^3\times (0, +\infty))$ for any $m\ge 1$. \\
(2) $(u_k, \nabla d_k)\rightarrow (u,\nabla d)$ weak$^*$ in $L^\infty([0,+\infty), L^3(\R^3))$.\\
Thus $(u,d)\in C^\infty(\R^3\times (0,+\infty),\R^3\times S^2)$ solves (\ref{LLF})$_1$, (\ref{LLF})$_2$, and (\ref{LLF})$_3$,
and the estimate (\ref{uniform_gradient_estimate}) holds.

Using the equation (\ref{LLF}), we can get that for any $0<T<+\infty$,
$$\sup_{k\ge 1}\Big\|(\partial_t u^k, \ \partial_t d^k)\Big\|_{L^\frac32([0,T), W^{-1,\frac32}(\R^3))}\le C(T)<+\infty.$$
This implies that $(u, d)\in C([0, T], L^3(\R^3)\times \dot{W}^{1,3}(\R^3))$ and $(u, d) |_{t=0}=(u_0,d_0)$. Applying
Corollary \ref{mono_l3} again, we conclude that $E_3(u(t),\nabla d(t))$ is monotone decreasing for $t\ge 0$.
The part of uniqueness can be proved as  in the step 6 of the proof of Theorem \ref{wellposedness} in \S 5,
which is omitted here.
The proof is complete. \qed

\medskip
We would like to mention applications of Theorem \ref{gwp} to the heat flow of harmonic maps and the Navier-Stokes
equation. \\
1) If $u\equiv 0$, then (\ref{LLF})$_3$ reduces to the heat flow of harmonic maps to $S^2$ for
$d:\R^3\times (0,+\infty)\to S^2$:
\begin{equation}\label{harmonic}
\begin{cases}
\partial_t d=\Delta d+|\nabla d|^2d &\ {\rm{in}}\ \R^3\times (0,+\infty)\\
\ \ d = d_0 &\ {\rm{on}}\ \ \R^3\times \{0\}.
\end{cases}
\end{equation}
2) If $d$ is a constant unit vector, then (\ref{LLF})$_1$ and (\ref{LLF})$_2$
reduce to the Navier-Stokes equation:
\begin{equation}\label{NSE}
\begin{cases}
\partial_t u+u\cdot\nabla u-\Delta u+\nabla P =0 & \ {\rm{in}}\ \R^3\times (0,+\infty)\\
\qquad \qquad \qquad \qquad \ \ \nabla\cdot u=0 & \ {\rm{in}}\ \R^3\times (0,+\infty)\\
\qquad \qquad \qquad \qquad \qquad\  u= u_0 & \ {\rm{on}}\ \R^3\times\{0\}.
\end{cases}
\end{equation}
The following properties follow directly from Theorem \ref{gwp}.
We would like to point out the observation of monotone decreasing property of the $L^3$-energy seems new.
\begin{rem} {\rm  1)
There exists $\epsilon_0 > 0$ such that if $d_0:\R^3\to S^2$ satisfies
$\displaystyle\int_{\R^3} |\grad d_0|^3\leq \epsilon_0^3$,
then there is a unique global solution $d:\R^3\times [0,+\infty)\to S^2$ of (\ref{harmonic})
such that $\displaystyle d\in C([0,+\infty), \dot{W}^{1,3}(\R^3, S^2))\cap C^\infty(\R^3\times (0, +\infty), S^2)$,
and $\displaystyle\int_{\R^3} |\grad d(t) |^3$ is monotone decreasing for $t\ge 0$.  \\
2) There exists $\epsilon_0 > 0$ such that if $u_0:\R^3\to \R^3$, with $\nabla\cdot u_0=0$,
satisfies $\displaystyle\int_{\R^3} |u_0|^3\leq \epsilon_0^3$,
then there is a unique, global solution $u:\R^3\times [0,+\infty)\to \R^3$ of (\ref{NSE})
such that $\displaystyle u\in C([0,+\infty), L^{3}(\R^3))\cap C^\infty(\R^3\times (0, +\infty), \R^3)$,
and $\displaystyle\int_{\R^3} |u(t) |^3$ is monotone decreasing for $t\ge 0$. }
\end{rem}

\section{Inequality of the local $L^3$-energy}

In this section, we will derive an inequality of the local $L^3$-energy for smooth solutions $(u, d):\mathbb R^3\times [0, T]\to
\mathbb R^3\times S^2$ for $0<T\le \infty$, of the system (\ref{LLF}). More precisely,
 we have

\begin{lemma}\label{local_ineq} There exists $C>0$ such that  for $0<T\le \infty$, if
$\displaystyle (u, d)\in C^\infty(\mathbb R^3\times [0, T),
\mathbb R^3\times S^2)\cap C([0,T), L^2(\R^3)\times \dot{W}^{1,2}(\R^3))$  is a smooth solution of the system (\ref{LLF}), then
\begin{equation}
  \begin{aligned}
    &\frac{d}{dt} \int_{\R^3} (|u|^3 + |\grad d|^3)\phi^2 + \int_{\R^3}\left(|\grad(|u|^{\frac32}\phi)|^2 + |\grad(|\grad d|^{\frac32}\phi)|^2\right) \\
    &\leq C \int_{\R^3} (|u|^3 + |\grad d|^3)|\grad \phi|^2 + CR^{-2}\sup_{y\in\mathbb R^3}\left(\int_{B_R(y)}|u|^3+|\nabla d|^3\right)^{\frac53} \\
    &+ C \left (\int_{\text{spt}\phi} |u|^3 +|\nabla d|^3\right )^{\frac23}
\int_{\R^3}\left(|\grad(|u|^{\frac32}\phi)|^2 + |\grad(|\grad d|^{\frac32}\phi)|^2\right),
  \end{aligned}
  \label{localEnergyInequality0}
\end{equation}
holds for any $\phi \in C_0^\infty(\R^3)$,
with $0\le \phi\le 1$, $\text{spt}\ \phi=B_R(x_0)$\footnote{Here spt$\phi$ denotes the support of $\phi$.}
 for some $R>0$ and $x_0\in\mathbb R^3$, and
$|\nabla\phi|\le 4R^{-1}$.
\end{lemma}
\begin{proof} We divide the proof into three  steps.

\medskip
\noindent{\bf Step 1}. Estimation of the local $L^3$-energy of $\nabla d$.
Differentiating \eqref{LLF}$_3$ with respect to $x$,  integrating against $\phi^2|\grad d|\grad d$
over $\R^3$,  and applying integration by parts, we have
\begin{equation}
  \label{l3LocDiffIneq1}
  \begin{aligned}
    &\frac{d}{dt} \int_{\R^3} |\grad d|^3 \phi^2
    + 3 \int_{\R^3} \grad^2 d :\grad
      (\phi^2|\grad d|\grad d) \\
    &\leq 3\int_{\R^3}|\grad d|^5\phi^2
    + 3\int_{\R^3} (u \cdot \grad d)\cdot \grad\cdot(\phi^2|\grad d| \grad d),
  \end{aligned}
\end{equation}
where we have used $|d| = 1$ and the following identity to obtain the first term on the right hand side:
\begin{equation*}
  \begin{aligned}
    \grad(|\grad d|^2 d)\cdot |\grad d|(\grad d)
    = \frac12\grad(|\grad d|^2)|\grad d| \grad (|d|^2)
    + |\nabla d|^3 \grad d \cdot \grad d= |\grad d|^5.
  \end{aligned}
\end{equation*}
For the second term on the left hand side of (\ref{l3LocDiffIneq1}), direct calculations using $|\nabla|\nabla d||\le|\nabla^2 d|$ and the H\"older inequality
imply
\begin{equation*}
  \label{est:diffd:1}
  \begin{aligned}
    \int_{\R^3} \grad^2d:\grad (\phi^2|\grad d|\grad d)
    &=\int_{\R^3} |\nabla d||\nabla^2 d|^2\phi^2+\int_{\R^3\cap\{|\nabla d|>0\}}(|\nabla d|^2 \nabla|\nabla d|\cdot\nabla\phi^2
+|\nabla d||\nabla|\nabla d||^2\phi^2)\\
&\ge \frac12 \int_{\R^3} |\nabla d||\nabla^2 d|^2\phi^2 -C\int_{\mathbb R^3}|\nabla d|^3|\nabla\phi|^2.
  \end{aligned}
\end{equation*}
For the second term on the right hand side of (\ref{l3LocDiffIneq1}), by the Cauchy inequality we have
\begin{equation*}
  \label{est:transd:1}
  \begin{aligned}
   \int_{\R^3} (u \cdot \grad d)\cdot \grad\cdot (\phi^2|\grad d| \grad d)
    &\leq 2\int_{\R^3} |u| |\grad d|^2 |\grad^2 d| \phi^2 + |u||\grad d|^3 \phi |\grad \phi| \\
&\leq \frac{1}{8}\int_{\R^3} |\grad d| |\grad^2 d|^2 \phi^2+C \left( \int_{\text{spt}\phi} |u|^{3} \right)^{\frac23}
    \left (\int_{\R^3} |\grad d|^{9} \phi^{6} \right)^{\frac13}\\
 &\quad+C \left (\int_{\text{spt}\phi} |u|^3 \right)^{\frac13}
    \left (\int_{\R^3} |\grad d|^{9}\phi^6 \right )^{\frac16}
    \left (\int_{\R^3} |\grad d|^{3}|\grad \phi|^2 \right)^{\frac12}\\
&\leq    \frac{1}{8}\int_{\R^3} |\grad d| |\grad^2 d|^2 \phi^2+C\int_{\R^3} |\grad d|^{3}|\grad \phi|^2\\
&\quad+C\left( \int_{\text{spt}\phi} |u|^{3} \right)^{\frac23}
    \left (\int_{\R^3} |\grad d|^{9} \phi^{6} \right)^{\frac13}.
  \end{aligned}
\end{equation*}
By the H\"older inequality and the Sobolev inequality, we have
\begin{equation*}\label{l5_of_d}
\left(\int_{\mathbb R^3}|\nabla d|^9\phi^6\right)^\frac13\lesssim \int_{\R^3} |\grad(|\grad d|^{\frac{3}{2}} \phi)|^2,
\ \
\int_{\mathbb R^3}|\nabla d|^5\phi^2\lesssim\left (\int_{\text{spt}\phi} |\grad d|^{3} \right)^{\frac{2}{3}}
 \int_{\R^3} |\grad(|\grad d|^{\frac{3}{2}} \phi)|^2.
\end{equation*}
Putting these estimates into (\ref{l3LocDiffIneq1})
yields
\begin{equation}
  \label{l3LocDiffIneq2}
  \begin{aligned}
    &\frac{d}{dt} \int_{\R^3} \phi^2|\grad d|^3
    + \int_{\R^3}|\grad^2 d|^2 |\grad d| \phi^2 \\
    &\lesssim\int_{\R^3} |\grad d|^3|\grad \phi|^2 +
      \left (\int_{\text{spt}\phi}|u|^{3} +|\nabla d|^3\right )^{\frac{2}{3}}
    \int_{\R^3} \left|\grad(|\grad d|^{\frac{3}{2}} \phi)\right|^2.
  \end{aligned}
\end{equation}

\noindent{\bf Step 2}. Estimation of the local $L^3$-energy of $u$.
Multiplying \eqref{LLF}$_1$ by $\phi^2|u|u$ and integrating over $\R^3$ yields
\begin{equation}
  \label{NSlocDiffIneq1}
  \begin{aligned}
    &\frac{d}{dt} \int_{\R^3} |u|^3\phi^2
   +3 \int_{\R^3} \grad u \cdot \grad (\phi^2|u|u)\\
    &\lesssim \int_{\R^3} |\nabla d||\nabla^2 d||u|^2\phi^2
    + \int_{\R^3} |\grad u| |u|^3\phi^2
    +\int_{\R^3} |P-c||\grad(\phi^2|u|u)|
  \end{aligned}
\end{equation}
where $c \in \R$ is a constant to be chosen later.

By the Cauchy inequality,  the H\"older inequality, and the Sobolev inequality, we have
\begin{equation*}
  \label{est:diffu:1}
  \begin{aligned}
\int_{\R^3} \grad u \cdot \grad (\phi^2|u|u) \geq \frac{1}{2} \int_{\R^3} |u||\grad u|^2 \phi^2
    - 4 \int_{\R^3} |u|^{3}|\grad \phi|^2,
  \end{aligned}
\end{equation*}
\begin{equation*}
  \label{est:transu:1}
  \begin{aligned}
    \int_{\R^3} | \grad u | |u|^3 \phi^2
    &\leq \frac{1}{4} \int_{\R^3} |u||\grad u|^2\phi^2
    + C\left (\int_{\text{spt}\phi} |u|^{3} \right)^{\frac{2}{3}} \int_{\R^3}\left|\grad (|u|^{\frac32} \phi)\right|^2,
  \end{aligned}
\end{equation*}
and
\begin{equation*}
\begin{aligned}
    &\int_{\R^3} |\grad d| |\grad^2 d| |u|^2\phi^2
    \leq \frac{1}{8} \int_{\R^3} |\grad d| |\grad^2 d|^2\phi^2 +C\int_{\mathbb R^3}|\nabla d||u|^4\phi^2\\
    &\leq \frac{1}{8} \int_{\R^3} |\grad d| |\grad^2 d|^2\phi^2 + C \left ( \int_{\text{spt}\phi} |\grad d|^3 \right)^{\frac13}
    \left (\int_{\text{spt}\phi} |u|^3 \right)^{\frac13} \left (\int_{\R^3} |u|^9 \phi^6 \right)^{\frac13}.
  \end{aligned}
\end{equation*}
For the last term on the right hand side of  ${\eqref{NSlocDiffIneq1}}$ we have
\begin{equation*}
  \begin{aligned}
    \int_{\R^3} |P-c| |\grad \cdot (|u|u\phi^2)|
    &\leq \frac{1}{8}\int_{\R^3} |u||\grad u|^2\phi^2
    +C\int_{\R^3} |P- c|^2|u|\phi^2 +
    C\int_{\R^3} |u|^3|\grad \phi|^2.
    \end{aligned}
    \label{localPressEst}
\end{equation*}
Putting these inequalities into (\ref{NSlocDiffIneq1}) we obtain
\begin{equation}
  \label{NS:energy:ineq}
  \begin{aligned}
    &\frac{d}{dt} \int_{\R^3}|u|^3\phi^2 + \int_{\R^3} |u||\grad u|^2\phi^2 \\
    &\leq  C \int_{\R^3}|u|^3|\grad \phi|^2 + \frac{1}{4} \int_{\R^3} |\grad d | |\grad^2 d|^2 \phi^2 + C \int_{\R^3} |P-c|^2 |u|\phi^2 \\
    &+ C \left (\int_{\text{spt}\phi} |u|^3 +|\nabla d|^3\right )^{\frac23} \int_{\R^3}(|\grad(|u|^{\frac32}\phi)|^2 + |\grad(|\grad d|^{\frac32}\phi)|^2).
  \end{aligned}
\end{equation}
Combining \eqref{l3LocDiffIneq2} with \eqref{NS:energy:ineq} yields
\begin{equation}
  \begin{aligned}
    &\frac{d}{dt} \int_{\R^3} (|u|^3 + |\grad d|^3)\phi^2 + \int_{\R^3}\left(|\grad(|u|^{\frac32}\phi)|^2 + |\grad(|\grad d|^{\frac32}\phi)|^2\right) \\
    &\leq C \int_{\R^3} (|u|^3 + |\grad d|^3)|\grad \phi|^2 + C\int_{\R^3}|u||P-c|^2\phi^2 \\
    &+ C \left (\int_{\text{spt}\phi} |u|^3 +|\nabla d|^3\right )^{\frac23}
\int_{\R^3}\left(|\grad(|u|^{\frac32}\phi)|^2 + |\grad(|\grad d|^{\frac32}\phi)|^2\right).
  \end{aligned}
  \label{localEnergyInequality1}
\end{equation}
\noindent{\bf Step 3}.  Estimation of the pressure function $P$.
By the H\"older inequality, we have
\begin{equation*}
  \int_{\R^3} |u||P-c|^2\phi^2 \leq \left (\int_{\text{spt}\phi} |u|^3 \right )^{\frac13} \left (\int_{\R^3} |P-c|^3\phi^3 \right )^{\frac23}.
\end{equation*}
We see that (\ref{localEnergyInequality0}) follows from (\ref{localEnergyInequality1}) and the estimate (\ref{p-estimate}) 
of Lemma \ref{L3:pressure:lemma} below. The proof is complete.
\end{proof}

\medskip
\begin{lemma}
\label{L3:pressure:lemma} Under the same assumptions as in Lemma \ref{local_ineq}, assume
that $\phi \in C_0^\infty(\R^3)$ satisfies $0\le\phi\le 1$,  $\text{spt}\ \phi =B_R(x_0)$
for some $x_0\in\mathbb R^3$,  and $|\nabla\phi|\le 2R^{-1}$. Then there exists $C>0$
such that for any $t\in (0, T)$ there is $c(t)\in \R$ so that the following estimate holds
\begin{equation}
  \label{p-estimate}
  \begin{aligned}
    \left ( \int_{\R^3} |P(t)- c(t)|^3 \phi^3 \right)^{\frac13}
    &\leq C \left(\int_{\text{spt}\phi}|u(t)|^3+|\nabla d(t)|^3\right)^\frac16
\left(\int_{\mathbb R^3}(|\nabla(|u(t)|^\frac32\phi)|^2+|\nabla(|\nabla d(t)|^\frac32\phi)|^2\right)^\frac12\\
&\quad+CR^{-1}
 \sup_{y \in \R^3} \left (\int_{B_R(y)} |u(t)|^3 + |\grad d(t)|^3\right )^{\frac23}.
  \end{aligned}
\end{equation}
\end{lemma}
\begin{proof} For simplicity, we write $(u,P,d)$ and $c$ for $(u(t), P(t), d(t))$ and $c(t)$ respectively.
  Since
  \begin{equation*}
    -\lap P  = \grad_{jk}^2 (g^{jk}), \ \ g^{jk} := u^ju^k + \grad_jd \cdot \grad_kd,
  \end{equation*}
 we have
  \begin{equation*}
    P = - {\mathbf R}_j{\mathbf R}_k(g^{jk})
  \end{equation*}
 where  ${\mathbf R}_j$ is the $j$-th Riesz transform on $\mathbb R^3$.
  Hence  we have
  \begin{equation}
    \begin{aligned}
      (P-c)\phi
      &= -{\bf R}_j {\bf R}_k(g^{jk})\phi  - c\phi\\
      &= -{\bf R}_j {\bf R}_k(g^{jk}\phi) - [\phi, {\bf R}_j{\bf R}_k](g^{jk}) -c\phi
    \end{aligned}
    \label{RRcommutator}
  \end{equation}
  where $[\phi, {\bf R}_j{\bf R}_k]$ is the commutator between $\phi$ and ${\bf R}_j{\bf R}_k$ given by
  \begin{equation*}
    [\phi, {\bf R}_j{\bf R}_k](f) = \phi\cdot {\bf R}_j{\bf R}_k(f) - {\bf R}_j{\bf R}_k(f\phi) , \ \ f \in C_0^\infty(\R^3).
  \end{equation*}
  We now estimate $\left[\phi,{\bf R}_j{\bf R}_k\right](g^{jk})$ as follows.
  \begin{equation*}
    \begin{aligned}
      &\left[\phi, {\bf R}_j{\bf R}_k\right](g^{jk})(x)\\
      &= \phi(x){\bf R}_j{\bf R}_k(g^{jk})(x) - {\bf R}_j{\bf R}_k(g^{jk}\phi)(x) \\
      &= \phi(x) \int_{\R^3} \frac{(x^j-y^j)(x^k-y^k)}{|x-y|^5} g^{jk}(y) -
      \int_{\R^3} \frac{(x^j-y^j)(x^k-y^k)}{|x-y|^5} \phi(y)g^{jk}(y)\\
      &= \int_{\R^3}\frac{(\phi(x)-\phi(y))(x^j-y^j)(x^k-y^k)}{|x-y|^5} g^{jk}(y).
    \end{aligned}
  \end{equation*}
 For any $x\in \text{spt}\ \phi=B_R(x_0)$, we have
  \begin{equation*}
    \begin{aligned}
      &\left[\phi, {\bf R}_j{\bf R}_k\right](g^{jk})(x) + c\phi(x) \\
      &= \int_{B_{2R}(x_0)}\frac{(\phi(x)-\phi(y))(x^j-y^j)(x^k-y^k)}{|x-y|^5} g^{jk}(y) dy + c\phi(x) \\
      &+ \phi(x) \left [\int_{\R^3 \setminus B_{2R}(x_0)}\frac{(x^j-y^j)(x^k-y^k)}{|x-y|^5} g^{jk}(y) dy + c \right] \\
      &=I(x) + II(x).
    \end{aligned}
  \end{equation*}
  For $I(x)$, we have that
  \begin{equation*}
    \begin{aligned}
      |I(x)|
      &\leq \int_{B_{2R}(x_0)} \frac{|\phi(x) - \phi(y)||x^j - y^j||x^k - y^k|}{|x-y|^5} |g^{jk}(y)| \\
      &\leq CR^{-1} \int_{\mathbb R^3}\frac{\chi_{B_{2R}(x_0)}(y)(|u|^2 + |\grad d|^2)(y)}{|x-y|^2} \\
      &= CR^{-1}\mathbf I_1\left ( (|u|^2+|\grad d|^2)\chi_{B_{2R}(x_0)} \right )(x),
      \end{aligned}
  \end{equation*}
  where  $\chi_{B_{2R}(x_0)}$ is the characteristic function
  of $B_{2R}(x_0)$, and $\mathbf I_1$ is the Riesz potential on $\mathbb R^3$ of order 1 given by
$$\mathbf I_1(f)(x)=\int_{\mathbb R^3} \frac{|f(y)|}{|x-y|^2}, \ x\in\mathbb R^3, \ \ \forall \ f\in L^1(\mathbb R^3).$$
Recall that by the Hardy-Littlewood-Sobolev inequality, $\mathbf I_1:L^{\frac32}(\R^3) \to L^{3}(\R^3)$ satisfies
  \begin{equation}
    \|\mathbf I_1(f)\|_{L^3(\R^3)} \lesssim \|f\|_{L^{\frac32}(\R^3)}.
  \end{equation}
  Hence we have
  \begin{equation}
    \begin{aligned}
      \|I\|_{L^3(\R^3)}
      &\lesssim R^{-1}\left\|\mathbf I_1\left ( (|u|^2+|\grad d|^2)\chi_{B_{2R}(x_0)} \right )\right\|_{L^{3}(\R^3)} \\
      &\lesssim R^{-1}\|(|u|^2+|\grad d|^2)\chi_{B_{2R}(x_0)}\|_{L^{\frac32}(\R^3)} \\
      &\lesssim R^{-1} \left (\int_{B_{2R}(x_0)} |u|^3 + |\grad d|^3 \right )^{\frac{2}{3}}\\
      &\lesssim R^{-1} \sup_{y\in\mathbb R^3}\left(\int_{B_{R}(y)} |u|^3 + |\grad d|^3 \right )^{\frac{2}{3}}.
     \end{aligned}
    \label{I-estimate}
  \end{equation}
  To estimate $II$, choose
  \begin{equation*}
    c= - \int_{\R^3\setminus B_{2R}(x_0)} \frac{(x_0-y)^j (x_0-y)^k}{|x_0-y|^5}g^{jk}(y).
  \end{equation*}
Note that
$$|c|\lesssim R^{-3}\sum_{j,k}
\int_{\R^3}|g^{jk}|\lesssim R^{-3} \left(\|u\|_{L^2(\R^3)}^2+\|\nabla d\|_{L^2(\R^3)}^2\right)
<+\infty.$$
Therefore we have
  \begin{equation*}
    \begin{aligned}
      |II(x)|
      &= \left | \phi(x) \int_{\R^3\setminus B_{2R}(x_0)}
        \left ( \frac{(x^j-y^j)(x^k - y^k)}{|x-y|^5} - \frac{(x_0-y)^j(x_0-y)^k}{|x_0-y|^5} \right ) g^{jk}(y)\right | \\
      &\lesssim R |\phi(x)| \int_{\R^3\setminus B_{2R}(x_0)} \frac{1}{|x-y|^4}(|u|^2 + |\grad d|^2)(y),
    \end{aligned}
  \end{equation*}
 where we have used the following inequality (see \cite{stein}):
  \begin{equation*}
    \left | \frac{(x^j-y^j)(x^k - y^k)}{|x-y|^5} - \frac{(x_0-y)^j(x_0-y)^k}{|x_0-y|^5} \right | \lesssim \frac{|x_0-x|}{|x_0-y|^4},
    \text{ for } x\in B_{R}(x_0) \text{ and } y \in \R^3 \setminus B_{2R}(x_0).
    \label{stdineq}
  \end{equation*}
  Thus we have
  \begin{equation*}
    \begin{aligned}
      |II|(x) &\lesssim R\int_{\R^3\setminus B_{2R}(x_0)} \frac{1}{|x_0-y|^4}(|u|^2 + |\grad d|^2)(y)\\
      &\lesssim R \sum_{k=2}^\infty \frac{1}{(kR)^4} \int_{B_{(k+1)R}(x_0)\setminus B_{kR}(x_0)} (|u|^2 + |\grad d|^2)\\
      &\lesssim \frac{1}{R^3} \left [\sum_{k=2}^\infty \frac{1}{k^2} \right ]
      \sup_{y \in \R^3} \int_{B_R(y)} (|u|^2 + |\grad d|^2) \\
      &\lesssim R^{-2}\sup_{y\in \R^3} \left ( \int_{B_R(y)} |u|^3 + |\grad d|^3 \right )^{\frac23}.
    \end{aligned}
  \end{equation*}
  Integrating $II$ over $B_{R}(x_0)$ we get
  \begin{equation}\label{II-estimate}
    \|II\|_{L^3(\R^3)} \lesssim R^{-1} \sup_{y \in \R^3} \left ( \int_{B_R(y)} |u|^3 + |\grad d|^3 \right )^{\frac23}.
  \end{equation}
  Additionally, we have
  \begin{equation}\label{gphi-estimate}
\begin{aligned}
    \left\|{\bf R}_j{\bf R}_k(g^{jk}\phi)\right\|_{L^3(\R^3)}
&\lesssim \left (\int_{\R^3} (|u|^6 + |\grad d|^6)\phi^3 \right)^{\frac13}\\
&\lesssim \left(\int_{\text{spt}\phi}|u|^3+|\nabla d|^3\right)^\frac16
\left(\int_{\mathbb R^3}(|u|^9+|\nabla d|^9)\phi^6\right)^\frac16 \\
&\lesssim \left(\int_{\text{spt}\phi}|u|^3+|\nabla d|^3\right)^\frac16
\left(\int_{\mathbb R^3}(|\nabla(|u|^\frac32\phi)|^2+|\nabla(|\nabla d|^\frac32\phi)|^2\right)^\frac12.
\end{aligned}
  \end{equation}
Combining  the estimates (\ref{I-estimate}) and (\ref{II-estimate}) with (\ref{gphi-estimate})
yields (\ref{p-estimate}). This completes the proof of Lemma \ref{L3:pressure:lemma}.
\end{proof}

\section{Regularity of Suitable Weak Solutions}

In this section, we will derive a priori estimates for smooth solutions to the system (\ref{LLF})
under a smallness condition for the $L^3$-norm of $(u,\nabla d)$. Since the method is flexible enough, it also
yields the smoothness for a subclass of suitable weak solutions to (\ref{LLF}). We present the result
in the context of suitable weak solutions to (\ref{LLF}). The notion of suitable
weak solutions was first introduced by Caffarelli-Kohn-Nirenberg \cite {CKN} in the context of incompressible Navier-Stokes equations. Here we adapt this notion to (\ref{LLF}), similar
to the definition given by Lin \cite{lin1} on the Navier-Stokes equation.

Let $0<T\le\infty$ and  $\Omega\subset\mathbb R^3$ be  a  bounded smooth domain.
\begin{defn}
  A triple of functions $(u,P,d): \Omega \times (0,T) \to \R^3 \times \R \times S^2$ is called a \emph{suitable  weak solution} to the system (\ref{LLF}) in $\Omega \times (0,T)$ if the following properties hold:
  \begin{enumerate}
    \item $u \in L_t^\infty L_x^2 \cap L_t^2H_x^1(\Omega \times (0,T))$, $P \in L^{\frac{3}{2}}(\Omega \times (0,T))$
      and $d \in L_t^2H_x^2(\Omega \times (0,T))$;
    \item $(u,P,d)$ satisfies the system (\ref{LLF}) in the sense of distributions; and
    \item $(u,P,d)$ satisfies the local energy inequality \eqref{locEnergy}. 
  \end{enumerate}
\end{defn}

Now we would like to point out that the class of smooth solutions  belongs to
the class of suitable weak solutions to the system (\ref{LLF}).
Let $\mathbb I_3$ denote the identity matrix of order $3$.

\begin{lemma} \label{local_energy_inequality}
Suppose that $(u,d)\in C^\infty(\Omega\times(0,T), \mathbb R^3\times \mathbb R\times S^2)$ is a solution of (\ref{LLF}) in $\Omega \times (0,T)$.
 Then for any nonnegative  $\phi \in C_0^\infty(\Omega \times (0,T))$, it holds that
  \begin{equation}
    \begin{aligned}
      2 \int_{\Omega \times (0,T)} \left ( |\grad u |^2 + | \lap d + |\grad d|^2 d|^2 \right )\phi
      &\leq \int_{\Omega \times (0,T)} \left ( |u|^2 + |\grad d|^2 \right )(\phi_t + \lap \phi) \\
      &+ \int_{\Omega \times (0,T)} (|u|^2 + |\grad d|^2 + 2P)u\cdot \grad \phi \\
      &+ 2 \int_{\Omega \times (0,T)} \left (\grad d \odot \grad d - |\grad d|^2 \mathbb I_3 \right):\grad^2 \phi \\
      &+ 2 \int_{\Omega \times (0,T)} \grad d \odot \grad d: u \otimes \grad \phi.
    \end{aligned}
    \label{locEnergy}
  \end{equation}
  \label{locEnergyLemma}
\end{lemma}
\begin{proof} Multiplying (\ref{LLF})$_1$ by $u\phi$ and integrating the resulting equation over
$\Omega\times (0, T)$ yields
  \begin{equation}
    \label{NS:suit:int}
    \begin{aligned}
      &\int_{\Omega\times(0,T)} u_t\cdot u\phi
      + \int_{\Omega\times(0,T)} (u\cdot \grad u)\cdot u\phi
      -\int_{\Omega\times(0,T)} \lap u\cdot u \phi
      + \int_{\Omega\times(0,T)} \grad P \cdot u\phi \\
      &=\int_{\Omega\times(0,T)} \grad d \odot \grad d:\nabla(u\phi).
    \end{aligned}
  \end{equation}
Applying integration by parts, the terms on the left hand side of (\ref{NS:suit:int})
can be estimated by
  \begin{equation*}
    \begin{aligned}
      \int_{\Omega \times (0,T)} u_t\cdot u\phi
      &= -\int_{\Omega \times (0,T)} \frac{1}{2}|u|^2 \phi_t, \\
      \int_{\Omega \times (0,T)} (u\cdot \grad u)\cdot u\phi
      &= -\int_{\Omega \times (0,T)} \frac12 |u|^2 u\cdot \grad\phi, \\
      \int_{\Omega \times (0,T)} \lap u\cdot u \phi
      &= -\int_{\Omega \times (0,T)} |\grad u|^2\phi
+ \int_{\Omega \times (0,T)} \frac12{|u|^2}  \lap \phi, \\
      \int_{\Omega \times (0,T)} \grad P \cdot u\phi
      &= - \int_{\Omega \times (0,T)} P(u \cdot \grad \phi).
    \end{aligned}
  \end{equation*}
  For the term on the right hand side of (\ref{NS:suit:int}),  we have
  \begin{equation*}
    \begin{aligned}
    \int_{\Omega \times (0,T)} \grad d \odot \grad d : \grad \cdot (u\phi)
   = \int_{\Omega \times (0,T)} \grad d \odot \grad d : \left [(\grad u) \phi + u \otimes \grad \phi \right].
    \end{aligned}
  \end{equation*}
  Putting these identities into \eqref{NS:suit:int} yields
  \begin{equation}
    \label{NS:suit:int:2}
    \begin{aligned}
      &\int_{\Omega \times (0,T)} - \frac{1}{2}|u|^2 (\phi_t  +\lap \phi) - (\frac12{|u|^2} + P)
    (u \cdot \grad \phi)
      + \int_{\Omega \times (0,T)} |\grad u|^2\phi \\
      &= \int_{\Omega \times (0,T)} (\grad d \odot \grad d) : \left [\phi \grad u + u \otimes \grad \phi \right].
    \end{aligned}
  \end{equation}
  Differentiating \eqref{LLF}$_3$ with respect to $x$ and integrating against $(\grad d)\phi$, we have
  \begin{equation}
    \label{dir:int:loc}
   \int_{\Omega \times (0,T)}(\grad d)_t : (\grad d)\phi
    + \int_{\Omega \times (0,T)} \grad(u \cdot \grad d):(\grad d)\phi
    = \int_{\Omega \times (0,T)} \grad \left [\lap d + |\grad d|^2d \right ]:(\grad d)\phi
  \end{equation}
For the first term on the left hand side of (\ref{dir:int:loc}),  we have
  \begin{equation*}
    \int_{\Omega \times (0,T)}(\grad d)_t : (\grad d)\phi = -\int_{\Omega \times (0,T)}\frac{1}{2}|\grad d|^2\phi_t.
  \end{equation*}
  Using \eqref{LLF}$_2$,  we can simplify the second term on the left hand side of (\ref{dir:int:loc}) into
  \begin{equation*}
    \begin{aligned}
      \int_{\Omega \times (0,T)}\grad(u \cdot \grad d):(\grad d)\phi
      &= \int_{\Omega \times (0,T)} u^j_\alpha d_j\cdot d_\alpha\phi  + \int_{\Omega \times (0,T)} u^j d_{j\alpha}\cdot d_\alpha\phi \\
      &= \int_{\Omega \times (0,T)} \grad u : \grad d \odot \grad d\phi
      + \int_{\Omega \times (0,T)} \frac12 u \cdot \grad (|\grad d|^2 ) \phi \\
      &= \int_{\Omega \times (0,T)} \grad u : \grad d \odot \grad d\phi
      - \int_{\Omega \times (0,T)} \frac12 (u\cdot\grad\phi) |\grad d|^2.
    \end{aligned}
  \end{equation*}
For the term on the right hand side of (\ref{dir:int:loc}),  differentiating $|d|=1$ gives
  \begin{equation*}
    \label{hm:id}
    \begin{aligned}
    \grad d \cdot d =0 \ \ {\rm{and}}\ \  \lap d \cdot d + |\grad d|^2=0.
    \end{aligned}
  \end{equation*}
  Thus, by integration by parts we have
  \begin{equation*}
    \label{dir:int:loc:1}
    \begin{aligned}
      \int_{\Omega \times (0,T)} \grad \left [\lap d + |\grad d|^2d \right ]\cdot\grad d\phi
      &= -\int_{\Omega \times (0,T)} \left [ \lap d + |\grad d|^2d \right ]\cdot
      \left [ \lap d\phi + \grad d \cdot \grad \phi \right] \\
      &= -\int_{\Omega \times (0,T)} |\lap d + |\grad d|^2d|^2\phi-
 \int_{\Omega \times (0,T)}\lap d\cdot (\grad d \cdot \grad \phi).
    \end{aligned}
  \end{equation*}
By integration by parts we have
  \begin{equation*}
    \label{dir:int:loc:2}
    \begin{aligned}
      - \int_{\Omega \times (0,T)} \lap d (\grad d \cdot \grad \phi)
      &= \int_{\Omega \times (0,T)} \left ( \grad d \odot \grad d \right ) : \grad^2 \phi
      - \int_{\Omega \times (0,T)} \frac12 |\grad d|^2\lap \phi\\
      &= \int_{\Omega \times (0,T)} \left ( \grad d \odot \grad d - |\grad d|^2\mathbb I_3 \right) : \grad^2 \phi
      + \int_{\Omega \times (0,T)} \frac12{|\grad d|^2} \lap \phi.
    \end{aligned}
  \end{equation*}
  Inserting these identities into \eqref{dir:int:loc} yields
  \begin{equation}
    \label{dir:int:loc:3}
    \begin{aligned}
      &\int_{\Omega \times (0,T)}
      [ -\frac{1}{2}|\grad d|^2(\phi_t + \lap \phi)- \frac{1}{2}|\grad d|^2(u \cdot \grad \phi)]
      + \int_{\Omega \times (0,T)} \grad u : \grad d \odot \grad d\phi \\
      &= \int_{\Omega \times (0,T)} \left ( \grad d \odot \grad d - |\grad d|^2\mathbb I_3 \right) : \grad^2 \phi
      -\int_{\Omega \times (0,T)} |\lap d + |\grad d|^2d|^2\phi.
    \end{aligned}
  \end{equation}
  Combining \eqref{NS:suit:int:2} with \eqref{dir:int:loc:3} yields \eqref{locEnergy}.
\end{proof}

\medskip
\begin{cor}
  Suppose that $(u,P,d):\Omega\times (0, T)\to \mathbb R^3\times \mathbb R\times S^2$ is a suitable weak solution of the system (\ref{LLF}) in $\Omega \times (0,T)$.
Then for any nonnegative $\phi \in C^\infty(\Omega \times (0,T))$ and  $0<t<T$, it holds
  \begin{equation}
    \begin{aligned}
      &\int_{\Omega \times \{t\}} (|u|^2 + |\grad d|^2)\phi
      + 2\int_{\Omega \times (0,t)} (|\grad u|^2 + |\lap d + |\grad d|^2d|^2) \phi \\
      &\leq \int_{\Omega \times (0,t)} (|u|^2 + |\grad d|^2)(\phi_t + \lap \phi)
      + \int_{\Omega \times (0,t)} (|u|^2 + |\grad d|^2 + 2P)u\cdot\grad \phi \\
      &+ 2\int_{\Omega \times (0,t)}(\grad d \odot \grad d - |\grad d|^2\mathbb I_3):\grad^2\phi
      + 2\int_{\Omega \times (0,t)} \grad d\odot \grad d:u \otimes \grad \phi.
      \label{HLWLocEnergyIneq}
    \end{aligned}
  \end{equation}
\end{cor}
\begin{proof}  For $\epsilon>0$, let $\eta_\epsilon\in C^\infty([0,t])$ be such that
$0\le \eta\le 1$,  \ $\eta=1$ in $[0, t-2\epsilon]$, and $\eta=0$ in $[t-\epsilon, t]$.
(\ref{HLWLocEnergyIneq}) follows by first applying
(\ref{locEnergy}), with $\phi$ replaced by $\eta_\epsilon(t)\phi(x,t)$, and then taking $\epsilon\rightarrow 0$.
\end{proof}

\medskip
Let $\mathcal C(3)>0$ denote the best Sobolev constant of $\mathbb R^3$:
\begin{equation}\label{sobolev}
\mathcal C(3):=\inf\Big\{\frac{\|\nabla f\|_{L^2(\mathbb R^3)}}{\|f\|_{L^6(\mathbb R^3)}}: \ 0\not=f\in C_0^\infty(\mathbb R^3)\Big\},
\end{equation}
and $\mathcal D(3)>0$ denote the constant  in  the following $W^{2,2}$-estimate:
\begin{equation} \label{w22constant}
\|\nabla^2 f\|_{L^2(B_1)}\le \mathcal  D(3)\|\Delta f\|_{L^2(B_1)}+C\|\nabla f\|_{W^{\frac12,2}(\partial B_1)}, \ \forall\ f\in W^{2,2}(B_1).
\end{equation}

For $z_0=(x_0,t_0)\in\mathbb R^3\times (0, T)$ and $r_0>0$, denote
$$B_{r_0}(x_0)=\{x\in\mathbb R^3 \ : \ |x-x_0|<r_0\}, \
P_{r_0}(z_0)=B_{r_0}(x_0)\times (t_0-r_0^2, t_0].$$
Now we are ready to prove the following $\epsilon_0$-regularity theorem.
\begin{thm}
  \label{smooth:solutions} For any $\delta>0$, there exists $\epsilon_0>0$ such that
  $(u,P,d):\Omega\times (0, T)\to\mathbb R^3\times\mathbb R\times S^2$ is a suitable
weak solution to \eqref{LLF}, and satisfies, for $z_0=(x_0,t_0)\in\Omega\times (0,T)$ and $P_{r_0}(z_0)\subset\Omega\times (0, T)$,
  \begin{equation}\label{eps0}
    \left ( r_0^{-2} \int_{P_{r_0}(z_0)} |u|^3 \right )^{\frac{1}{3}}
    + \left ( r_0^{-2} \int_{P_{r_0}(z_0)} |P|^{\frac{3}{2}} \right )^{\frac{2}{3}}
    + \left ( r_0^{-2} \int_{P_{r_0}(z_0)} |\grad d|^3 \right )^{\frac{1}{3}}
\leq \epsilon_0,
  \end{equation}
  and
  \begin{equation}\label{l3infinity}
   \Big\|\grad d\Big\|_{L^\infty_tL^3_x(P_{r_0}(z_0))} < \frac{1-\delta}{\mathcal{C}(3)\mathcal D(3)},
  \end{equation}
  then $(u,d) \in C^{\infty}(P_{\frac{r_0}4}(z_0), \mathbb R^3\times S^2)$, and the following estimate holds:
  \begin{equation}
  \label{cm:estimate}
  \|(u,d)\|_{C^m(P_{\frac{r_0}4}(z_0))} \leq C(m,r_0, \epsilon_0), \ \forall\ m\ge 0.
  \end{equation}
\end{thm}

\medskip
The crucial ingredient to prove Theorem \ref{smooth:solutions} is the  following decay lemma, 
which is analogous to that of the Navier-Stokes equations by \cite{lin1} and \cite{ESV}.
\begin{lemma}
  \label{decayThm} For any $\delta>0$, there exist $\epsilon_0 > 0$ and $\theta_0 \in (0,\frac{1}{2})$ such that
if $(u,P,d): \Omega\times (0,T)\to\mathbb R^3\times\mathbb R\times S^2$ is a suitable weak solution of (\ref{LLF}),
and satisfies, for $z_0=(x_0,t_0)\in\Omega\times (0,T)$ and $P_{r_0}(z_0)\subset\Omega\times (0, T)$,
both (\ref{eps0}) and (\ref{l3infinity}),
  then it holds that
  \begin{equation}
    \label{decay:ineq}
    \begin{aligned}
      &\left [
        \left ( (\theta_0 r_0)^{-2} \int_{P_{\theta_0r_0}(z_0)} |u|^3 \right )^{\frac{1}{3}}
        +\left ( (\theta_0 r_0)^{-2} \int_{P_{\theta_0r_0}(z_0)} |P|^{\frac{3}{2}} \right )^{\frac{2}{3}}
        +\left ( (\theta_0 r_0)^{-2} \int_{P_{\theta_0r_0}(z_0)} |\grad d|^3 \right )^{\frac{1}{3}}
      \right ] \\
      &\leq \frac{1}{2}
      \left [
        \left ( r_0^{-2} \int_{P_{r_0}(z_0)} |u|^3 \right )^{\frac{1}{3}}
        +\left ( r_0^{-2} \int_{P_{r_0}(z_0)} |P|^{\frac{3}{2}} \right )^{\frac{2}{3}}
        +\left ( r_0^{-2} \int_{P_{r_0}(z_0)} |\grad d|^3 \right )^{\frac{1}{3}}
      \right ].
    \end{aligned}
  \end{equation}
\end{lemma}

\begin{proof}  By the invariance of (\ref{LLF}) under translations and parabolic dilations,
it suffices to consider the case that $z_0=(0,0)$ and $r_0=1$.
We will prove the Lemma by contradiction. Suppose that the conclusion were false. Then there would exist
$\delta_0>0$ such that for any $\theta\in (0,1)$ there are a sequence
of suitable weak solutions $(u_i, P_i, d_i)$ of (\ref{LLF}) in $P_1$, that satisfy
  \begin{equation}
    \label{epsContra1}
      \left (\int_{P_1} |u_i|^3 \right )^{\frac{1}{3}}
    + \left (\int_{P_1} |P_i|^{\frac{3}{2}} \right )^{\frac{2}{3}}
    + \left (\int_{P_1} |\grad d_i|^3 \right )^{\frac{1}{3}} = \epsilon_i \to 0,
  \end{equation}
\begin{equation}\label{l3infinity1}
\left\|\nabla d_i\right\|_{L^\infty_tL^3_x(P_1)}\le \frac{1-\delta_0}{\mathcal C(3)\mathcal D(3)},
\end{equation}
 and
  \begin{equation}
    \label{epsContra2}
    \begin{aligned}
      &\left [
        \left ( \theta^{-2} \int_{P_\theta} |u_i|^3 \right )^{\frac{1}{3}}
        +\left ( \theta^{-2} \int_{P_\theta} |P_i|^{\frac{3}{2}} \right )^{\frac{2}{3}}
        +\left ( \theta^{-2} \int_{P_\theta} |\grad d_i|^3 \right )^{\frac{1}{3}}
      \right ] \\
      &> \frac{1}{2}
      \left [
        \left (\int_{P_1} |u_i|^3 \right )^{\frac{1}{3}}
        +\left (\int_{P_1} |P_i|^{\frac{3}{2}} \right )^{\frac{2}{3}}
        +\left (\int_{P_1} |\grad d_i|^3 \right )^{\frac{1}{3}}
      \right ].
    \end{aligned}
  \end{equation}
  Now we define the blow-up sequence $(v_i, Q_i, e_i):P_1 \to \R^3 \times \R \times \R^3$ by
  \begin{equation*}
    v_i(z) = \frac{u_i(z)}{\epsilon_i}, \ Q_i(z) = \frac{P_i(z)}{\epsilon_i}, \ e_i(z) = \frac{d_i(z)-(d_i)_1}{\epsilon_i},
\ \ z\in P_1,
  \end{equation*}
  where $\displaystyle (d_i)_1 = \frac{1}{|P_1|}\int_{P_1}d_i$ is the average of $d_i$ over $P_1$. \\ 
Then $(v_i,Q_i,e_i)$ satisfy the following equations in $P_1$:
  \begin{equation}\begin{cases}
    \begin{aligned}
      \partial_t v_i - \lap v_i + \grad Q_i &= -\epsilon_i [v_i \cdot \grad v_i + \grad\cdot(\grad e_i \odot \grad e_i)], \\
      \grad \cdot v_i &= 0, \\
      \partial_t e_i - \lap e_i &= \epsilon_i [|\grad e_i|^2 d_i - v_i\cdot \grad e_i].
    \end{aligned}
\end{cases}
    \label{scaledEqs}
  \end{equation}
  It follows from \eqref{epsContra1} and (\ref{epsContra2}) that for any $\theta \in (0,\frac{1}{2})$,
  \begin{equation}
    \label{blowup}
    \left (\int_{P_1} |v_i|^3 \right )^{\frac{1}{3}}
    +\left (\int_{P_1} |Q_i|^{\frac{3}{2}} \right )^{\frac{2}{3}}
    +\left (\int_{P_1}|\nabla e_i|^{3} \right )^{\frac{1}{3}}
    = 1,
  \end{equation}
and
  \begin{equation}
    \label{blowup:ineq}
    \left (\theta^{-2}\int_{P_\theta} |v_i|^3 \right )^{\frac{1}{3}}
    +\left (\theta^{-2}\int_{P_\theta} |Q_i|^{\frac{3}{2}} \right )^{\frac{2}{3}}
    +\left (\theta^{-2}\int_{P_\theta} |\nabla e_i|^{3} \right )^{\frac{1}{3}}
    >\frac{1}{2}.
  \end{equation}
Applying the $W^{2,1}_{\frac32}$-estimate to the equation (\ref{scaledEqs})$_3$, we have
that $\nabla^2 e_i\in L^\frac32(P_{\frac78})$ and
\begin{equation}\label{21-estimate}
\left\|\nabla^2 e_i\right\|_{L^\frac32(P_\frac78)}\lesssim\left (\|v_i\|_{L^3(P_1)}^2+\|\nabla e_i\|_{L^3(P_1)}^2\right)
\leq C.
\end{equation}
By the Fubini Theorem and (\ref{21-estimate}), we may assume that
\begin{equation} \label{fubini}
\int_{\partial B_\frac34\times [-(\frac34)^2, 0]}|\nabla^2 e_i|^\frac32
\le C\int_{P_\frac78}|\nabla^2 e_i|^\frac32\le C.
\end{equation}

Since $(u_i,Q_i,d_i)$ satisfies \eqref{HLWLocEnergyIneq} in $P_1$,  by choosing suitable test functions $\phi$
we have that
  \begin{equation}
    \label{loc:ener:blowup:ineq}
    \begin{aligned}
      &\sup_{-\left (\frac{3}{4} \right)^2 \leq t \leq 0} \int_{B_{\frac34}} (|u_i|^2 + |\grad d_i|^2)
      + \int_{P_{\frac34}} (|\grad u_i|^2 + |\lap d_i + |\grad d_i|^2d_i|^2) \\
      &\qquad\qquad\leq C \int_{P_1}(|u_i|^2 + |\grad d_i|^2) + (|Q_i| + |u_i|^2  + |\grad d_i|^2)|u_i|.
    \end{aligned}
  \end{equation}
  Rescaling \eqref{loc:ener:blowup:ineq}, applying \eqref{blowup}, and using the H\"older inequality,  we have
  \begin{equation}
    \label{loc:ener:blowup:ineq:rescaled}
    \begin{aligned}
      &\sup_{-\left (\frac{3}{4} \right)^2 \leq t \leq 0} \int_{B_{\frac34}} (|v_i|^2 + |\grad e_i|^2)
      + \int_{P_{\frac34}} (|\grad v_i|^2 + \left |\lap e_i + \epsilon_i |\grad e_i|^2d_i \right |^2) \\
      &\qquad\qquad\leq C \int_{P_1}(|v_i|^2 + |\grad e_i|^2) + (|Q_i| + \epsilon_i|v_i|^2  + \epsilon_i|\grad e_i|^2)|v_i|
      \leq C.
    \end{aligned}
  \end{equation}
 By the $W^{2,2}$-estimate (\ref{w22constant}) and the Sobolev inequality,   we have
  \begin{equation*}
\begin{aligned}
    \int_{B_{\frac34}} |\grad^2 e_i|^2 &\leq \mathcal D^2(3)
\int_{B_{\frac34}} |\lap e_i|^2+C\|\nabla e_i\|_{W^{\frac12, 2}(\partial B_\frac34)}^2\\
&\leq \mathcal D^2(3)
\int_{B_{\frac34}} |\lap e_i|^2+C\|\nabla^2 e_i\|_{W^{2, \frac32}(\partial B_\frac34)}^2,
\end{aligned}
  \end{equation*}
so that, by integrating over $\displaystyle t\in [-(\frac34)^2,0]$ and applying (\ref{fubini}), it holds that
  \begin{eqnarray}
   \int_{P_{\frac34}} |\grad^2 e_i|^2
&\leq&
 \mathcal D^2(3)\int_{P_{\frac34}} |\lap e_i|^2+C\int_{-(\frac34)^2}^0
\|\nabla^2 e_i\|_{W^{2, \frac32}(\partial B_\frac34)}^2\nonumber\\
&\le& C+\mathcal D^2(3)\int_{P_{\frac34}} |\lap e_i|^2. \label{w22ineq20}
  \end{eqnarray}
By the point-wise identity
$\displaystyle|\lap e_i|^2 =|\lap e_i +\epsilon_i |\grad e_i|^2d_i|^2 + \epsilon_i^2|\grad e_i|^4,$
we have
  \begin{equation}\label{4-22}
      \int_{P_{\frac34}} |\Delta e_i|^2=
      \int_{P_{\frac34}} |\lap e_i + \epsilon_i|\grad e_i|^2d_i|^2
      + \epsilon_i^2\int_{P_{\frac34}} |\grad e_i|^4.
  \end{equation}
 By the H\"older inequality, the Young inequality, and the Sobolev inequality, we have
  \begin{eqnarray}
  &&\|\grad e_i\|_{L^4(B_\frac34)}^4\leq\|\nabla e_i\|_{L^3(B_\frac34)}^2
\|\nabla e_i\|_{L^6(B_\frac34)}^2\nonumber\\
&&\leq\|\nabla e_i\|_{L^3(B_\frac34)}^2\Big(
\|\nabla e_i-(\nabla e_i)_{\frac34}\|_{L^6(B_\frac34)}^2+\|(\nabla e_i)_{\frac34}\|_{L^6(B_\frac34)}\Big)^2\nonumber\\
&&\leq(1+\delta_0)^2\mathcal C^2(3) \|\grad e_i\|_{L^3(B_\frac34)}^2\|\grad^2 e_i\|_{L^2(B_\frac34)}^2
+C(\delta_0)\|\nabla e_i\|_{L^3(B_\frac34)}^2\|\nabla e_i\|_{L^2(B_\frac34)}^2, \label{w22ineq10}
  \end{eqnarray}
where $(\nabla e_i)_{\frac34}$ is the average of $\nabla e_i$ over $B_\frac34$.
Integrating  (\ref{w22ineq10}) over $\displaystyle t\in [-(\frac34)^2, 0] $ yields
  \begin{eqnarray}
    \label{w22ineq:2}
    \epsilon_i^2\int_{P_\frac34} |\grad e_i|^4
    &\leq& (1+\delta_0)^2\mathcal C^2(3)\|\grad d_i \|_{L^\infty_tL^3_x(P_\frac34)}^2
    \int_{P_\frac34} |\grad^2 e_i|^2 \nonumber\\
&&+C(\delta_0)(\sup_{-(\frac34)^2\le t\le 0}\int_{B_\frac34}|\nabla d_i|^2)\|\nabla e_i\|_{L^3(P_\frac34)}^2\nonumber\\
&\leq & C(\delta_0)+(1+\delta_0)^2\mathcal C^2(3)\|\grad d_i \|_{L^\infty_tL^3_x(P_\frac34)}^2
    \int_{P_\frac34} |\grad^2 e_i|^2.
  \end{eqnarray}
  Inserting the estimate \eqref{w22ineq:2} first into \eqref{4-22} and then (\ref{w22ineq20}), we obtain
  \begin{eqnarray}
    \label{w22ineq:3}
   && \left [1 - (1+\delta_0)^2\mathcal{C}^2(3)\mathcal D^2(3)\|\grad d_i \|_{L^\infty_tL^3_x(P_1)}^2\right]
    \int_ {P_{\frac34}}|\grad^2 e_i|^2 \nonumber\\
   && \leq  C(\delta_0)+C
    \int_{P_\frac34} |\lap e_i + \epsilon_i|\grad e_i|^2d_i|^2\le C(\delta_0).
  \end{eqnarray}
  Therefore, by applying \eqref{loc:ener:blowup:ineq:rescaled} to \eqref{w22ineq:3},  we have
  \begin{equation}
    \label{w22ineq:4}
    \int_{P_\frac34} |\grad^2 e_i|^2 \leq C(\delta_0).
  \end{equation}
  Combining the estimates \eqref{loc:ener:blowup:ineq:rescaled} and \eqref{w22ineq:4},
  we obtain
  \begin{equation}
    \int_{P_\frac12} |Q_i|^{\frac{3}{2}}
    + \sup_{t\in [-\frac14, 0]}\int_{B_\frac12} (|v_i|^2 + |\grad e_i|^2)
    +\int_{P_\frac12} (|\grad v_i|^2 + |\grad^2 e_i|^2) \leq C.
    \label{PhalfEnergy}
  \end{equation}
  We may assume, after taking possible subsequences, that
  \begin{equation*}
\begin{cases}
    \begin{aligned}
      &Q_i \to Q \text{ weakly in } L^{\frac{3}{2}}(P_{\frac{1}{2}}), \\
      &v_i \to v \text{ strongly in } L^2(P_\frac12),\ \grad v_i \to \grad v \text{ weakly in } L^2(P_{\frac{1}{2}}), \\
      &e_i \to e\ {\rm{and}}\ \grad e_i \to \grad e \text{ strongly in } L^2(P_\frac12), \ \grad^2 e_i \to \grad^2 e \text{ weakly in }
L^2(P_{\frac{1}{2}}).
    \end{aligned}
\end{cases}
  \end{equation*}
  Sending $i$ to $\infty$ in the equation \eqref{scaledEqs} yields that $(v,Q,e)$ satisfies in $P_\frac12$
  \begin{equation}\label{limit_eqn}
\begin{cases}
    \begin{aligned}
      \partial_t v - \lap v + \grad Q &= 0, \\
      \grad \cdot v &= 0, \\
      \partial_t e- \lap e &= 0.
    \end{aligned}
\end{cases}
  \end{equation}
  Using the Sobolev inequality and interpolations, we see that \eqref{PhalfEnergy} gives
  \begin{equation}\label{vqe_bound}
    \int_{P_\frac12} |v|^3 + |Q|^{\frac{3}{2}} + |\grad e|^{3} \leq C.
  \end{equation}
Hence, by the standard estimates on the linear Stokes equation and the heat equation, we have
that for any $\theta \in (0,\frac12)$, it holds
  \begin{equation}
    \label{cacciopoli:est}
 \theta^{-2} \int_{P_\theta} (|v|^3 + |\grad e|^3)
  \leq C  \theta^{3} \int_{P_\frac12}(|v|^3 +|\grad e|^3)\le C\theta^3,\
\ \theta^{-2}\int_{P_\theta}|Q|^\frac32 \le C\theta\int_{P_\frac12}|Q|^\frac32\le C\theta.
  \end{equation}
In order to reach a contradiction, we need to show that $(v_i, Q_i, e_i)$ converges  to $(v,Q,e)$ strongly in
$L^3(P_\frac25)$. To do so, we  recall the following Lemma (see \cite{temam}).
  \begin{lemma}\label{aubin}
    Let $X_0 \subset X \subset X_1$ be Banach spaces such that $X_0$ is compactly embedded in $X$, $X$ is continuously embedded
    in $X_1$, and $X_0$, $X_1$ are reflexive.  Then for $1 < \alpha_0, \alpha_1 < \infty$,
    \begin{equation*}
      \Big\{u \in L^{\alpha_0}(0,T;X_0) : \ \partial_t u \in L^{\alpha_1}(0,T;X_1) \Big \}
      \text{ is compactly embedded in }
      L^{\alpha_0}(0,T;X).
    \end{equation*}
  \end{lemma}

Now we have the following claims.\\
{\it Claim} 1.  $v_i \to v$ strongly in $L^2(P_{\frac25})$.
  From \eqref{PhalfEnergy} and interpolation inequalities, we have
  \begin{equation*}\begin{cases}
    \label{bounds:v_i:grade_i}
    \begin{aligned}
      &\|v_i\|_{L^{\frac{10}3}(P_{\frac12})}+\|v_i\|_{L^\infty_tL^2_x(P_\frac12)}+\|\nabla v_i\|_{L^2(P_{\frac12})}  \leq C, \\
      &\|\nabla e_i\|_{L^{\frac{10}3}(P_\frac12)}+\|\nabla e_i\|_{L^\infty_tL^2_x(P_\frac12)}
+\|\grad^2 e_i\|_{L^2(P_{\frac12})}  \leq C.
    \end{aligned}
\end{cases}
  \end{equation*}
  So by the H\"older inequality, we have
  \begin{equation*}
\begin{cases}
    \begin{aligned}
      \int_{P_{\frac12}} |v_i \cdot \grad v_i|^{\frac 54}
      &\leq \left (\int_{P_{\frac12}} |v_i |^{\frac{10}3} \right )^{\frac38} \left (\int_{P_{\frac12}} |\grad v_i|^{2} \right )^{\frac58} \leq C ,\\
      \int_{P_{\frac12}} |\grad \cdot (\grad e_i \odot \grad e_i)|^{\frac54}
      &\leq \left ( \int_{P_{\frac12}} |\grad^2 e_i|^{2}\right )^{\frac58}
\left (\int_{P_{\frac12}}|\grad e_i|^{\frac{10}3} \right)^{\frac38} \leq C.
    \end{aligned}
\end{cases}
  \end{equation*}
  These inequalities imply
  \begin{equation}
    \label{stokesRHS}
    \Big\|
      \epsilon_i \left [v_i \cdot \grad v_i + \grad \cdot (\grad e_i \odot \grad e_i) \right ] \Big\|_{L^{\frac54}(P_{\frac12})} \leq C.
  \end{equation}
  By \eqref{stokesRHS} and the $W^{2,1}_{\alpha}$-estimate of the  linear Stokes equation, we have 
  \begin{equation}
    \label{bound:partialv_i}
    \Big\|\partial_t v_i \Big\|_{L^{\frac54}(P_{\frac25})} \leq C.
  \end{equation}
Hence $\{v_i\}$ is bounded in
  \begin{equation*}
    \mathbf{X}_1 = \left \{u \in L_t^2H_x^1(P_{\frac 25})\  :\  \partial_t u \in L_t^{\frac54} L_x^{\frac54}(P_{\frac25}) \right \}.
  \end{equation*}
Since $\mathbf{X}_1$ is compactly embedded in $L^2_t L^2_x(P_{\frac 25})$ by Lemma \ref{aubin}, we conclude that
$v_i \to v \text{ strongly in } L^2(P_{\frac25})$.

\smallskip
 \noindent {\it Claim} 2.  $\grad e_i \to \grad e$ strongly in $L^2(P_{\frac25})$.
Using \eqref{bounds:v_i:grade_i} and the H\"older inequality we have
  \begin{equation*}
    \begin{aligned}
      \left\|v_i \cdot \grad e_i \right\|_{L^{\frac{20}{11}}(P_{\frac 12})}
\leq \left\|v_i\right\|_{L^{\frac{10}3}(P_{\frac12})}\left\|\grad e_i\right\|_{L^4(P_{\frac12})}  \leq C,
    \end{aligned}
  \end{equation*}
so that
  \begin{equation}
    \label{heatRHS:bound}
    \left \| |\grad e_i|^2d_i + v_i\cdot \grad e_i \right \|_{L^{\frac{20}{11}}(P_{\frac12})} \leq C.
  \end{equation}
Hence the $W^{2,1}_{\alpha}$-estimate for the heat equation implies
  \begin{equation}
    \label{bound:partialgrade_i}
    \left\|\partial_t \grad e_i \right\|_{L^{\frac{20}9}_tW^{-1,\frac{20}9}_x(P_{\frac25})} \leq C.
  \end{equation}
  By \eqref{bounds:v_i:grade_i} and \eqref{bound:partialgrade_i}, we have $\{\grad e_i\}$ is bounded in
  \begin{equation*}
    \mathbf{X}_2 = \left \{u \in L_t^2H_x^1(P_{\frac25})\  :\  \partial_t u \in L_t^{\frac{20}9} W_x^{-1,\frac{20}9}(P_{\frac25}) \right \},
  \end{equation*}
  and so by Lemma \ref{aubin}, we have that $\grad e_i \to \grad e \text{ strongly in } L^2(P_{\frac25}).$
  It is easy to see that by interpolations, the claims imply that
  \begin{equation}
    \label{strongL3:conv}
    v_i \to v,\  \grad e_i \to \grad e \text{ strongly in } L^3(P_{\frac25}).
  \end{equation}
  From \eqref{strongL3:conv} and (\ref{cacciopoli:est}), we conclude that for any $\theta \in (0,\frac14)$ and $i$ sufficiently large,
  \begin{equation}
    \label{v:grad:e:L3:control}
    \theta^{-2} \int_{P_\theta} |v_i|^3 + |\grad e_i|^3
    \leq \theta^{-2} \int_{P_\theta} |v|^3 + |\grad e|^3 + o(1) \leq C\theta^3.
  \end{equation}
  Finally using the estimate \eqref{pressure:est:1} below, with $\tau = \theta$ and $r =\frac12$, we have
that for any $0<\theta<\frac14$,
  \begin{equation*}
    \theta^{-2} \int_{P_{\theta}} |P_i|^{\frac32}
    \leq C
    \left [
      \theta^{-2} \int_{P_\frac12} (|u_i|^3 + |\grad d_i|^3) +\theta \int_{P_\frac12}|P_i|^{\frac32}
      \right ].
  \end{equation*}
After scaling, this implies that for any $0<\theta<\frac14$,
\begin{equation}\label{qi-estimate}
 \theta^{-2} \int_{P_{\theta}} |Q_i|^{\frac32}
   \leq C
    \left [
      \theta^{-2} \epsilon_i^\frac32\int_{P_\frac12} (|v_i|^3 + |\grad e_i|^3) +\theta \int_{P_\frac12}|Q_i|^{\frac32}
      \right ]
\leq C(\epsilon_i^\frac32\theta^{-2}+\theta).
  \end{equation}
  Combining \eqref{v:grad:e:L3:control} and \eqref{qi-estimate}, we have that for sufficiently large $i = i(\theta)$,
  \begin{equation*}
  \theta^{-2} \int_{P_{\theta}} (|v_i|^3 + |\grad e_i|^3 + |Q_i|^{\frac32}) \leq C\theta.
  \end{equation*}
 This contradicts \eqref{blowup:ineq}, if we choose $\theta\in (0,\frac14)$ sufficiently small.
\end{proof}

\renewcommand{\div}{\operatorname{div}}
\renewcommand{\tilde}{\widetilde}
The next Lemma gives the estimate of pressure function, which is needed in the proof of Lemma \ref{decayThm}.
\begin{lemma}
  \label{pressure:est}
   Suppose that $(u,P,d)$ is a suitable weak solution of (\ref{LLF}) on $P_1$.
   Then for any $0<r\le 1$ and $\tau \in (0,\frac{r}{2})$, it holds that
   \begin{equation}
     \label{pressure:est:1}
     \frac{1}{\tau^2} \int_{P_\tau} |P|^{\frac32}
     \leq C
     \left [
       \left ( \frac{r}{\tau} \right )^2 \frac{1}{r^2}
       \int_{P_r} (|u - u_{r}(t)|^3 + |\grad d|^3)
       + \left ( \frac{\tau}{r} \right ) \frac{1}{r^2}
       \int_{P_r} |P|^{\frac32}
     \right ],
   \end{equation}
   where $\displaystyle u_{r}(t) = \frac{1}{|B_r|} \int_{B_r} u(x,t) $ for $- r^2 \leq t \leq 0$.
   In particular, it holds that
   \begin{equation}
     \label{pressure:est:2}
     \begin{aligned}
       \frac{1}{\tau^2} \int_{P_\tau} |P|^{\frac32}
       &\leq C
       \left ( \frac{r}{\tau} \right )^2
       \left ( \sup_{- r^2\leq t \leq 0} \frac{1}{r} \int_{B_r} |u|^2 \right )^{\frac34}
       \left ( \frac{1}{r} \int_{P_r} |\grad u|^2 \right )^{\frac34} \\
       &+C
       \left [
         \left ( \frac{r}{\tau} \right )^2 \frac{1}{r^2}
         \int_{P_r} |\grad d|^3
         +\left ( \frac{\tau}{r} \right ) \frac{1}{r^2}
         \int_{P_r} |P|^{\frac32}
       \right ].
     \end{aligned}
   \end{equation}
\end{lemma}
\begin{proof} By scaling, it suffices to consider the case $r = 1$.  Using the equation (\ref{LLF})$_2$, we have
  \begin{equation*}
    \begin{aligned}
      \div \div \left [(u - u_{1}(t)) \otimes (u-u_{1}(t)) \right]
      &= \grad_j \grad_i \left ((u - u_{1}(t))^i (u-u_{1}(t))^j \right )\\
      &= \grad_j \left ((u - u_{1}(t))^i \grad_i (u-u_{1}(t))^j \right)= \grad_j \left ((u - u_{1}(t))^i \grad_i u^j \right)\\
      &= \grad_j (u - u_{1}(t))^i \grad_i u^j
      + (u - u_{1}(t))^i \grad_i \grad_j u^j\\
      &= (\grad_j u^i) (\grad_i u^j)
      = \grad_j \grad_i (u^iu^j)
      = \div \div (u \otimes u).
    \end{aligned}
  \end{equation*}
  Taking the divergence of \eqref{LLF}$_1$, this yields
  \begin{equation}\label{p-equation}
    \begin{aligned}
      \lap P
      &= - \div \div \left [(u - u_{1}(t)) \otimes (u - u_{1}(t)) + \grad d \odot \grad d \right ].
    \end{aligned}
  \end{equation}
  Let $\eta \in C_0^{\infty}(\R^3)$ be a cut-off function of $B_{\frac12}$, i.e. $0\leq \eta \leq 1$, $\eta \equiv 1$
  on $B_{\frac12}$, $\eta \equiv 0$ outside $B_1$, and $|\grad \eta| \leq C$.  Define $\tilde{P}$ by
  \begin{equation*}
    \tilde{P}(x,t)
    = - \int_{\R^3} \grad_y^2G(x-y):\eta^2(y)
    \left ((u - u_{1}(t)) \otimes (u - u_{1}(t)) + \grad d \odot \grad d \right)(y,t),
  \end{equation*}
  where $G$ is the fundamental solution of the Laplace equation on $\R^3$.  We have
  \begin{equation*}
    \begin{aligned}
      \lap \tilde{P} = \div \div \left ((u - u_{1}(t)) \otimes (u - u_{1}(t)) + \grad d \odot \grad d \right)\ \ {\rm{in}}\ \ \mathbb R^3.
    \end{aligned}
  \end{equation*}
  By the Calderon-Zygmund $L^p$-theory we have
  \begin{equation*}
    \begin{aligned}
      \int_{B_\tau} |\tilde{P}(t)|^{\frac32}
      &\leq \int_{\R^3} |\tilde{P}(t)|^{\frac32}
      \lesssim \int_{\R^3} \eta^3 \left |(u - u_{1}(t)) \otimes (u - u_{1}(t)) + \grad d \odot \grad d \right|^{\frac32} \\
      &\lesssim \int_{B_1}  \left ( |u - u_{1}(t)|^{3} + |\grad d|^3 \right).
    \end{aligned}
  \end{equation*}
  Integrating this inequality over $t \in (-\tau^2,0)$ yields
  \begin{equation}
    \frac{1}{\tau^2} \int_{P_\tau}|\tilde{P}|^{\frac32} \leq \frac{C}{\tau^2} \int_{P_1}(|u - u_{1}(t)|^3+|\grad d|^3).
    \label{tildepineq}
  \end{equation}
  Since the function $Q:= P - \tilde{P} \in L^{\frac32}(P_1)$ satisfies
  \begin{equation*}
    \lap Q(t) = 0\ \  \text{ in }\ \  B_{\frac12}, \ \forall\  t\in [-\frac14,0],
  \end{equation*}
  we have by the Harnack inequality that for any $0<\tau<\frac12$, 
  \begin{equation*}
    \begin{aligned}
      \frac{1}{\tau^2}\int_{B_\tau} |Q|^{\frac32}
 &\leq C \tau \int_{B_{\frac12}} |Q|^{\frac32}
   \leq C \tau \left [ \int_{B_{1}} |P|^{\frac32} +  \int_{B_{1}} |\tilde{P}|^{\frac32} \right ]  \\
      &\leq C \tau \left [ \int_{B_{1}} |P|^{\frac32} +  \int_{B_{1}} |u-u_{1}(t)|^{3}+|\grad d|^3 \right ].
    \end{aligned}
  \end{equation*}
  Integrating this inequality over $t \in [-\tau^2,0]$ implies
  \begin{equation}
    \label{qineq}
    \frac{1}{\tau^2} \int_{P_\tau} |Q|^{\frac32}
    \leq C\tau \left[ \int_{P_1} |P|^{\frac32} + \int_{P_1} |u-u_{1}(t)|^{3}+|\grad d|^3 \right ].
  \end{equation}
  It is now readily seen that \eqref{pressure:est:1} follows by adding the inequalities \eqref{tildepineq} and \eqref{qineq}.
  Using interpolation and the Sobolev inequality, we have
  \begin{equation}
    \label{sobolev:est:average}
    \begin{aligned}
      \int_{B_1} |u - u_{1}|^3
      &\leq  C\left ( \int_{B_1} |u|^2 \right )^{\frac{3}{4}}
     \left ( \int_{B_1} |\grad u|^2 \right )^{\frac{3}{4}}.
    \end{aligned}
  \end{equation}
Inserting \eqref{sobolev:est:average} into \eqref{pressure:est:1} yields \eqref{pressure:est:2}.
\end{proof}

Continuing to iterate the above process, we have
\begin{cor}
\label{iteratedEst}
Under the same assumptions as Lemma \ref{decayThm}, there exists $\alpha \in (0,1)$ such that for any
$z_1\in P_{\frac{r_0}2}(z_0)$ and $0 < \tau < r<\frac{r_0}2$,
it holds 
\begin{equation}
  \label{iter:decay:ineq}
  \begin{aligned}
    &\left ( \frac{1}{\tau^2} \int_{P_\tau (z_1)} |u|^3 \right )^{\frac13}
    + \left ( \frac{1}{\tau^2} \int_{P_\tau(z_1)} |P|^{\frac32} \right )^{\frac23}
    + \left ( \frac{1}{\tau^2} \int_{P_\tau(z_1)} |\grad d|^{3} \right )^{\frac13} \\
    &\leq
    \left ( \frac{\tau}{r} \right )^\alpha
    \left [
      \left ( \frac{1}{r^2} \int_{P_r(z_1)} |u|^3 \right )^{\frac13}
    + \left ( \frac{1}{r^2} \int_{P_r(z_1)} |P|^{\frac32} \right )^{\frac23}
    + \left ( \frac{1}{r^2} \int_{P_r(z_1)} |\grad d|^{3} \right )^{\frac13}
    \right ].
  \end{aligned}
\end{equation}
\end{cor}
\begin{proof} Set $r_1=\frac{r_0}2$ and $\epsilon_1=2^{\frac83}\epsilon_0$.
Then it follows from  (\ref{eps0}) and (\ref{l3infinity}) that for any $z_1\in P_{\frac{r_0}2}(z_0)$,
both (\ref{eps0}) and (\ref{l3infinity})  also hold for $(u,P,d)$ with $z_0, r_0$ and $\epsilon_0$
replaced by $z_1, r_1$ and $\epsilon_1$ respectively.
For $0<\rho<r_1$, define $\Phi(\rho)$ by
\begin{equation*}
  \Phi(\rho) := \left ( \frac{1}{\rho^2} \int_{P_\rho(z_1)}|u|^3 \right )^{\frac13}
  + \left ( \frac{1}{\rho^2} \int_{P_\rho(z_1)} |P|^{\frac32} \right )^{\frac23}
  + \left ( \frac{1}{\rho^2} \int_{P_\rho(z_1)} |\grad d|^{3} \right )^{\frac13}. \\
\end{equation*}
Then  applying Lemma \ref{decayThm} for $(u,P,d)$ on
$P_{r_1}(z_1)$, there exists $\theta_0 \in (0,\frac12)$
such that for any $0<r\le r_1$, it holds that
\begin{equation*}
  \label{Phi:ineq}
  \Phi(\theta_0 r) \leq \frac{1}{2}\Phi(r) \leq \frac{1}{2}\epsilon_1.
\end{equation*}
Iterating \eqref{Phi:ineq} $k$-times, $k\ge 1$,  yields
\begin{equation*}
  \label{iter:decay:ineq:1}
  \Phi(\theta_0^kr) \leq 2^{-k} \Phi(r).
\end{equation*}
It is well known that this implies that there exists $\alpha\in (0,1)$ such that
for any $0<\tau<r\le r_1$, $\displaystyle\Phi(\tau)\le (\frac{\tau}{r})^\alpha\Phi(r)$.  Therefore \eqref{iter:decay:ineq} holds.
\end{proof}

\medskip
\noindent{\bf Proof of Theorem \ref{smooth:solutions}}.
We will now prove the smoothness of $(u,d)$ in $P_{\frac{r_0}4}(z_0)$ by
 the estimate \eqref{iter:decay:ineq}.
The idea is based on the Riesz potential estimates
between Morrey spaces, that is analogous to those of Huang-Wang \cite{HW} and Lin-Wang \cite{LW}.

First, let's recall the notion of Morrey spaces on $\mathbb R^3\times\mathbb R$, equipped with the parabolic metric ${\bf \delta}$: 
$$\delta\Big((x,t), (y,s)\Big)=\max\Big\{|x-y|, \sqrt{|t-s|}\Big\}, \ \forall\ (x,t), \ (y,s)\in\mathbb R^3\times\mathbb R.$$
For any open set $U \subset \R^{3+1}$, $1 \leq p < +\infty$, and $0 \leq \lambda \leq 5$,
define the {Morrey Space} $M^{p,\lambda}(U)$ by
\begin{equation}
  M^{p,\lambda}(U) :=
  \left \{
    v\in L^p_{\rm{loc}}(U):
   \left \|v\right\|^p_{M^{p,\lambda}(U)} \equiv \sup_{z \in U, r>0} r^{\lambda - 5} \int_{P_r(z) \cap U} |v|^p < \infty
  \right \}.
  \label{morreyDef}
\end{equation}
By Corollary \ref{iteratedEst} we have that for some $\alpha\in (0,1)$,
\begin{equation}
  u, \ \grad d \in M^{3,3(1-\alpha)} \left (P_{\frac{r_0}2}(z_0) \right).
  \label{morreyugradd}
\end{equation}
Write the equation (\ref{LLF})$_3$ as
\begin{equation}\label{director1}
  \partial_t d - \lap d = f, \ {\rm{with}}\ \ f:= (|\grad d|^2d - u \cdot \grad d).
\end{equation}
By \eqref{morreyugradd}, we see that
\begin{equation*}
  f \in M^{\frac32,3(1-\alpha)} \left (P_{\frac{r_0}2}(z_0) \right ).
\end{equation*}
As in  \cite{LW} and  \cite{HW},  let $\eta \in C_0^\infty(\R^{3+1})$ be a cut-off function of $P_{\frac{r_0}2}(z_0)$:
$0\le\eta\le 1$, \ $\eta\equiv 1$ in $P_{\frac{r_0}2}(z_0)$, and $|\partial_t\eta|+|\nabla^2\eta|\le Cr_0^{-2}$.
Set $w = \eta^2 d$.  Then we have
\begin{equation}
  \label{Fdef}
  \partial_t w - \lap w =F, \ \ F:= \eta^2 f + (\partial_t \eta^2 - \lap \eta^2)d - 2 \grad \eta^2 \cdot \grad d.
\end{equation}
It is easy to check that $F \in M^{\frac32,3(1-\alpha)}(\R^{3+1})$ and satisfies the estimate
\begin{equation}\label{F-estimate}
  \begin{aligned}
    \Big\| F \Big\|_{M^{\frac32,3(1-\alpha)}(\R^{3+1})} &\leq C\left [ 1 + \|f\|_{M^{\frac32,3(1-\alpha)}(P_{\frac{r_0}2}(z_0))}\right]
\le C(1+\epsilon_0).
  \end{aligned}
\end{equation}
Let $\Gamma(x,t)$ denote the fundamental solution of the heat equation on $\mathbb R^3$. Then
by  the Duhamel formula for (\ref{Fdef}) and  the estimate (see \cite{HW} Lemma 3.1):
$$|\nabla\Gamma|(x,t)\lesssim \frac{1}{\delta^{4}((x,t), (0,0))}, \ \forall (x,t)\not=(0,0),$$
we have
\begin{equation}
  \label{riesz:pot:est:1}
  \begin{aligned}
    |\grad w(x,t)|
    &\leq \int_0^t \int_{\R^3} |\grad \Gamma(x-y,t-s)||F(y,s)|
    \leq C \int_{\R^4} \frac{|F(y,s)|}{\delta^4((x,t),(y,s))}:
    = C I_1(|F|)(x,t),
  \end{aligned}
\end{equation}
where $\mathcal I_\beta $ is the Riesz potential of order $\beta$ on $\mathbb R^4$ ($\beta\in [0,5]$),
defined by
\begin{equation}
  \label{riesz:pot:def}
  \mathcal{I}_\beta(g) = \int_{\R^{4}} \frac{|g(y,s)|}{\delta((x,t),(y,s))^{5-\beta}}, \ \ g \in L^p(\mathbb R^{4}).
\end{equation}
Applying the Riesz potential estimates (see \cite{HW} Theorem 3.1),  we conclude that
$\grad w \in M^{\frac{3(1-\alpha)}{1-2\alpha}, 3(1-\alpha)}(\R^{4})$ and
\begin{equation}
  \label{morrey:gradw:est}
  \Big\|\grad w \Big\|_{M^{\frac{3(1-\alpha)}{1-2\alpha}, 3(1-\alpha)}(\R^{4})}  \lesssim
 \Big\|F\Big\|_{M^{\frac32,3(1-\alpha)}(\R^{4})}
\lesssim \left [ 1 + \|f\|_{M^{\frac32,3(1-\alpha)}(P_{\frac{r_0}2}(z_0))}\right]\lesssim (1+\epsilon_0).
\end{equation}
Choosing $\alpha\uparrow \frac12$ and using
$\lim_{\alpha \uparrow \frac{1}{2}}\frac{3(1-\alpha)}{1-2\alpha} = +\infty$, we can conclude that
for any $1<q<\infty$,
$\nabla w\in L^q(P_{r_0}(z_0))$ and
\begin{equation}
  \label{gradd:L_m:m.gt.1}
  \Big\|\grad w \Big\|_{L^q\left ( P_{r_0}(z_0) \right )}\leq C(q, r_0,\epsilon_0).
\end{equation}
Since $(d-w)$ solves
$$\partial_t(d-w)-\Delta(d-w)=0 \ {\rm{in}}\ P_{\frac{r_0}2}(z_0),$$
it follows from the standard estimate on the heat equation
that for any $1<q<+\infty$, $\nabla d\in L^q(P_{\frac{r_0}4}(z_0))$ and
\begin{equation}
\label{estimate_d}
  \Big\|\grad d \Big\|_{L^q( P_{\frac{r_0}4}(z_0))}\leq C(q, r_0,\epsilon_0).
\end{equation}

Now we proceed with the estimation of $u$.  Let $v:\mathbb R^3\times [0,+\infty)\to\mathbb R^3$
solve the Stokes equation:
\begin{equation}
\begin{cases}
  \label{aux:stokes}
  \begin{aligned}
    \partial_t v - \lap v + \grad Q&=  -\grad \cdot[\eta^2 (\grad d \odot \grad d + u \otimes u)] &\text{ in } \R^3 \times (0,\infty),\\
    \grad\cdot v &= 0 &\text{ in } \R^3 \times (0,\infty), \\
    v(\cdot,0) &= 0 &\text{ in } \R^3.
  \end{aligned}
\end{cases}
\end{equation}
By using the Oseen kernel (see Leray \cite{Leray}), an estimate for $v$, similar to \eqref{riesz:pot:est:1},  can be given by
\begin{equation}
  \label{riesz:pot:est:2}
  |v(x,t)| \leq C \int_0^t\int_{\R^{3}} \frac{|X(y,s)|}{\delta((x,t),(y,s))^{3+1}}
\leq C \mathcal I_1(|X|)(x,t),  \ (x,t)\in\mathbb R^3\times (0,+\infty),
\end{equation}
where $X = \eta^2(\grad d \odot \grad d + u \otimes u )$. As above, we can check
that $X \in M^{\frac32,3(1-\alpha)}(\R^{4})$ and
\begin{equation*}
 \Big\|X\Big\|_{M^{\frac32,3(1-\alpha)}(\R^{4})}
\leq C \left [ \|\grad d\|^2_{M^{3,3(1-\alpha)}(P_{\frac{r_0}2}(z_0))} + \|u\|^2_{M^{3,3(1-\alpha)}(P_{\frac{r_0}2}(z_0))} \right ].
\end{equation*}
Hence, by  \cite{HW} Theorem 3.1, we have that $v \in M^{\frac{3(1-\alpha)}{1-2\alpha}, 3(1-\alpha)}(\R^{4})$,  and
\begin{equation}
  \label{morrey:v:est}
   \Big\|v \Big\|_{M^{\frac{3(1-\alpha)}{1-2\alpha}, 3(1-\alpha)}(\R^{4})}
\leq C\Big\|X\Big\|_{M^{\frac32,3(1-\alpha)}(\R^{4})}\le C \left [ \|\grad d\|^2_{M^{3,3(1-\alpha)}(P_{\frac{r_0}2}(z_0))} + \|u\|^2_{M^{3,3(1-\alpha)}(P_{\frac{r_0}2}(z_0))} \right ].
\end{equation}
By sending $\alpha\uparrow \frac12$, (\ref{morrey:v:est}) implies that for any $1<q<+\infty$,
$v \in L^q\left (P_{r_0}(z_0)\right )$ and
\begin{equation}\label{estimate_q}
\Big\|v \Big\|_{L^q\left (P_{r_0}(z_0)\right )}\le C(q, r_0, \epsilon_0).
\end{equation}
Note that $(u-v)$ satisfies the linear homogeneous Stokes equation in $P_{\frac{r_0}2}(z_0)$:
\begin{equation*}
\partial_t (u-v) - \lap (u-v) + \grad(P - Q) = 0, \ \grad \cdot (u-v) = 0\ \  \text{ in } \ \ P_{\frac{r_0}2}(z_0).
\end{equation*}
It is well-known that $(u-v) \in L^{\infty}(P_{\frac{r_0}4}(z_0))$.  Therefore we conclude that for any
$1<q<+\infty$, $u\in L^q(P_{\frac{r_0}4}(z_0))$, and
\begin{equation}
  \label{estimate_u}
 \Big\| u \Big\|_{L^q( P_{\frac{r_0}4}(z_0))}\le C(q, r_0, \epsilon_0).
\end{equation}
It is now standard that by (\ref{estimate_d}) and (\ref{estimate_u}), and estimates for the linear parabolic equation and
the linear Stokes equation,  $\displaystyle (u,d)\in C^\infty(P_{\frac{r_0}4}(z_0), \mathbb R^3\times S^2)$ and the estimate (\ref{cm:estimate}) holds.
\qed

\medskip
\section{Existence of $L^3_{\rm{uloc}}$-solutions and Proofs of Theorem
\ref{wellposedness}}

In this section, we will prove our main result -- Theorem \ref{wellposedness}. 

\medskip
\noindent{\bf Proof of Theorem \ref{wellposedness}}.  First, observe that by the scaling invariance of (\ref{LLF}), $(u,P,d):\mathbb R^3\times [0, T)\to
\R^3\times\R\times S^2$ solves (\ref{LLF}) under the initial condition
$(u_0,d_0)$ if and only if for any $\lambda>0$, $(u^\lambda, P^\lambda, d^\lambda): \mathbb R^3\times [0, T^\lambda)\to
\R^3\times\R\times S^2$ solves (\ref{LLF}) under the initial condition
$(u_0^\lambda,d_0^\lambda)$. Here
$$\displaystyle T^\lambda={\lambda^{-2}}T,
\ (u_0^\lambda(x), d^\lambda_0(x))=(\lambda u_0(\lambda x), d_0(\lambda x))
\ \ \ {\rm{for}}\ \ x\in\R^3;$$
and
$$\left(u^\lambda(x,t), P^\lambda(x,t), d^\lambda(x,t)\right)
=\left(\lambda u(\lambda x, \lambda^2 t), \lambda^2 P(\lambda x, \lambda^2 t),
d(\lambda x, \lambda^2 t)\right)
\ \ \ {\rm{for}}\ \ (x,t)\in\R^3\times [0, T^\lambda).$$
Therefore it suffices to prove Theorem \ref{smooth:solutions} for $R=1$. We divide
the proof into six steps. 

\medskip
\noindent{\bf Step 1}. {\it Approximation of $(u_0,d_0)$ by smooth initial data}.  We summarize this step into the following
lemma.
\begin{lemma}\label{approx} For a sufficiently small $\epsilon_0>0$,
let $(u_0, d_0):\R^3\to \R^3\times S^2$, with $u_0\in L^3_{\rm{uloc}}(\R^3)$ divergence free
and $(d_0-e_0)\in L^3(\R^3)$ for some $e_0\in S^2$, satisfy
\begin{equation}\label{small_condition1}
 |||(u_0,\nabla d_0)|||_{L^3_{1}(\R^3)}\leq \epsilon_0.
\end{equation}
Then there exist a large constant $C_0>0$ and
$$\displaystyle\{(u_0^k, d_0^k)\}\subset C^\infty(\R^3,\R^3\times S^2)\cap \bigcap_{p=2}^3(L^p(\R^3,\R^3)\times \dot{W}^{1,p}(\R^3,S^2))$$
such that the following properties hold:\\
(i) $\nabla\cdot u_0^k=0$ in $\R^3$ for all $k\ge 1$.\\
(ii) As $k\rightarrow \infty$,
\begin{equation}\label{local_convergence}
(u_0^k,d_0^k)\rightarrow (u_0, d_0) \ {\rm{and}}\ \nabla d_0^k\rightarrow \nabla d_0 
\ {\rm{in}}\ L^p_{\rm{loc}}(\R^3) \ {\rm{for}}\ p=2, 3.
\end{equation}
(iii) There exists  $k_0>1$ such that for any $k\ge k_0$,
\begin{equation} \label{uniform_l3loc}
|||(u_0^k,\nabla d_0^k)|||_{L^3_{1}(\R^3)}\leq C_0\epsilon_0.
\end{equation}
\end{lemma}

We assume Lemma \ref{approx} for the moment and  continue the proof of Theorem \ref{wellposedness}. 
By modifying the proof of the local existence Theorem 3.1 of Lin-Lin-Wang \cite{LLW}\footnote{For $K>0$ and $0<\alpha<1$, first choose the solution space
\begin{equation*}
\begin{aligned}
X_T&=\Big\{(u,d): \R^3\times [0,T]\to\R^3\times\R^3: \nabla\cdot u=0,
\ \nabla^2f, \partial_t f\in C_b(\R^3\times [0, T])\cap
C^\alpha(\R^3\times [0,T]), \\
&\qquad (u,d)|_{t=0}=(u_0^k, d_0^k), \ \|(u-u_0^k, d-d_0^k)\|_{C^{2,1}_\alpha(\R^3\times [0,1])}\le K\Big\},
\end{aligned}
\end{equation*}
then follow the fixed point argument as in \cite{LLW} with slight modifications, one can obtain the local existence of smooth solutions.
}, we can conclude that
there exist $0<T_k< +\infty$ and smooth solutions $(u^k,P^k,  d^k):\R^3\times
[0, T_k]\to \R^3\times\R\times S^2$ of (\ref{LLF}), under the initial condition
$\displaystyle (u^k, d^k)|_{t=0}=(u_0^k, d_0^k)$.
Observe that by applying the proof of Lemma \ref{local_energy_inequality} with $\phi\equiv 1$,
the following energy inequality holds:
\begin{equation}\label{global_energy_ineq}
\int_{\R^3}(|u^k(t)|^2+|\nabla d^k(t)|^2)
+2\int_0^t\int_{\R^3}|\nabla u^k|^2+|\Delta d^k+|\nabla d^k|^2 d^k|^2
=\int_{\R^3}(|u^k_0|^2+|\nabla d^k_0|^2), \ 0\le t\le T_k.
\end{equation}
In particular, we have that $(u_k, d_k)\in C([0, T_k], L^2(\R^3)\times \dot{W}^{1,2}(\R^3))$.

\medskip
\noindent{\bf Step 2}. {\it Uniform lower bounds of $T_k$}. To see this, we first need to show

\smallskip
\noindent{\it Claim}. There exists $\tau_0>0$  such that
if $T_k$ is the maximal time interval for the
smooth solutions $(u^k, d^k)$ obtained in step 1, then $T_k\ge\tau_0$, and
\begin{equation}\label{small_condition4}
  \sup_{0\leq t \leq \tau_0} |||(u^k(t), \grad d^k(t))|||_{L^3_\frac12(\R^3)}^3\leq 2C_0^3\epsilon_0^3.
\end{equation}
To see (\ref{small_condition4}), note that (\ref{uniform_l3loc}) implies that
there exists a maximal time $t_k^*\in (0, T_k]$ such that
  \begin{equation} \label{small_condition5}
    \sup_{0\leq t \leq t_k^*}  |||(u^k(t), \grad d^k(t))|||_{L^3_\frac12(\R^3)}^3\le 2C_0^3\epsilon_0^3.
  \end{equation}
 Hence
  \begin{equation}\label{small_condition6}
 |||(u^k(t_k^*), \grad d^k(t_k^*))|||_{L^3_\frac12(\R^3)}^3=2C_0^3\epsilon_0^3.
  \end{equation}
By a simple covering argument, we see that (\ref{small_condition5}) implies
\begin{equation}\label{small_condition7}
\sup_{0\le t\le t_k^*}\sup_{x\in\mathbb R^3}\int_{B_1(x)}(|u^k(t)|^3+|\nabla d^k(t)|^3)\le C\epsilon_0^3.
\end{equation}
  For any fixed $x_0 \in \R^3$, let $\phi_0 \in C_0^\infty(\R^3)$ be a cut-off function of
$B_\frac12(x_0)$:
  \begin{equation*}
    0 \leq \phi_0 \leq 1, \ \phi_0 \equiv 1 \text{ on } B_\frac12(x_0), \ \phi_0 \equiv 0
    \text { outside } B_1(x_0) \text{, and } |\grad \phi_0|\leq 4.
  \end{equation*}
  For convenience, we set for $0\le t\le t_k^*$,
  \begin{equation}
    \mathcal E_3^k(\phi_0; (x_0,t)):= \int_{\R^3}\left[|u^k(t)|^3 + |\grad d^k(t)|^3\right]\phi_0^2.
  \end{equation}
  Then by \eqref{localEnergyInequality1} and (\ref{small_condition7}) we have that
for any $0\le t\le t_k^*$,
  \begin{equation}
    \label{lifespan:energy:1}
    \begin{aligned}
      &\frac{d}{dt} \mathcal E_3^k(\phi_0; (x_0, t)) +
(1-C\epsilon_0^2)\int_{\R^3}[|\grad(|u^k(t)|^{\frac32}\phi_0)|^2
+ |\grad(|\grad d^k(t)|^{\frac32}\phi_0)|^2] \\
      &\leq C \int_{\R^3} (|u^k(t)|^3 + |\grad d^k(t)|^3)|\grad \phi_0|^2 + C
\sup_{y\in\R^3}\left( \int_{B_1(y)}|u^k(t)|^3+|\nabla d^k(t)|^3\right)^\frac53\\
      &\leq C \epsilon_0^3
      + C\epsilon_0^5\le C\epsilon_0^3.
    \end{aligned}
  \end{equation}
 Integrating (\ref{lifespan:energy:1}) with respect to $t \in [0, t_k^*]$ yields
  \begin{equation}\label{lifespan2}
    \begin{aligned}
  &\mathcal E_3^k(\phi_0; (x_0, t_k^*)) + (1-C\epsilon_0^2) \int_0^{t_k^*}\int_{\R^3}[|\grad(|u^k|^{\frac32}\phi_0)|^2 + |\grad(|\grad d^k|^{\frac32}\phi_0)|^2] \\
      &\qquad\qquad\qquad\qquad\qquad\qquad\leq C \epsilon_0^3t_k^* +\mathcal E_3^k(\phi_0; (x_0, 0))\le C \epsilon_0^3t_k^* +C_0^3\epsilon_0^3,
    \end{aligned}
  \end{equation}
where we have used (\ref{uniform_l3loc}) in the last step. Therefore if $\epsilon_0>0$ is chosen such that
$1-C\epsilon_0^2\ge 0$, then
(\ref{lifespan2}) implies
 \begin{equation*}\label{lifespan3}
 \mathcal E_3^k(\phi_0; (x_0, t_k^*)) \le C \epsilon_0^3t_k^* +C_0^3\epsilon_0^3.
  \end{equation*}
Taking the supremum of $\mathcal E_3^k(\phi_0; (x_0, t_k^*))$ over $x_0 \in \R^3$, we obtain
\begin{equation*}
  \label{lifespan4}
  2C_0^3\epsilon_0^3=|||(u^k(t^*_k), \nabla d^k(t_k^*)|||_{L^3_\frac12(\R^3)}^3\leq\sup_{x_0\in\R^3}E_3^k(\phi_0; (x_0, t_k^*))
\le C\epsilon_0^3 t_k^*+C_0^3\epsilon_0^3.
  \end{equation*}
This clearly implies that there exists $\tau_0>0$  such that $T_k\ge t_k^*\ge \tau_0$. By the definition
of $t_k^*$, we also see that the estimate (\ref{small_condition4}) holds.

\medskip
\noindent{\bf  Step 3}. {\it Uniform estimation of $(u^k, d^k)$}. Note that $P^k$ satisfies
$$\Delta P^k=-{\rm{div}}^2(u^k\otimes u^k+\nabla d^k\odot\nabla d^k)\ \ \ {\rm{in}}\ \ \ \R^3.$$
It follows from (\ref{global_energy_ineq}),  (\ref{small_condition4}) and Lemma \ref{L3:pressure:lemma} that
\begin{equation}\label{pressure-estimate10}
\sup_{0\le t\le\tau_0}\sup_{x\in\R^3}\left\|P^k(t)-c_x^k(t)\right\|_{L^3(B_1(x))}\le C\epsilon_0,
\end{equation}
where $c_x^k(t)\in\R$ depends on both $x\in\R^3$ and $t\in [0,\tau_0]$.
By (\ref{small_condition4}) and (\ref{pressure-estimate10}), we see that
for any $x_0\in\R^3$, $(u^k, P^k-c_{x_0}^k, d^k)$ satisfies the conditions
of Theorem \ref{smooth:solutions} in $\displaystyle P_{\sqrt{\tau_0}}(x_0,\tau_0):=B_{\sqrt{\tau_0}}(x_0)\times [0,\tau_0]$.
Hence by Theorem \ref{smooth:solutions} we obtain that
$(u^k, d^k)\in C^\infty(\R^3\times (0, \tau_0), \R^3\times S^2)$, and
\begin{equation}\label{uniform_estimate10}
\sup_k \left\|(u^k, \nabla d^k)\right\|_{C^m(\R^3\times [\delta,\tau_0])}\le C(m, \delta, \epsilon_0)
\end{equation}
holds for any $0<\delta<\frac{\tau_0}2$ and $m\ge 0$.

\medskip
\noindent{\bf  Step 4}. {\it Passage to the limit}. Based on the estimates of $(u^k, d^k)$, we may assume, after taking subsequences,
that $\displaystyle (u,d)\in \bigcap_{0<\delta<\tau_0}C^\infty_b(\R^3\times [\delta, \tau_0], \R^3\times S^2)$, with $\displaystyle (u,\nabla d)\in 
\ L^\infty([0, \tau_0], L^3_{\rm{uloc}}(\R^3)$, such that
$$(u^k, \nabla d^k)\rightarrow (u,\nabla d) \ {\rm{weakly\ in }}\  L^3(\R^3\times [0,\tau_0]),
(u^k,d^k)\rightarrow (u, d) \ {\rm{in}}\ C^m(B_R\times [\delta, \tau_0]), \ \forall\ m\ge 0, R>0, \delta<\tau_0.$$
Sending $k\rightarrow\infty$ in (\ref{small_condition7}) yields
$$\sup_{0\le t\le \tau_0}\left\|(u,\nabla d)\right\|_{L^3_1(\R^3)}\le C\epsilon_0.$$
We can check from (\ref{LLF}) and (\ref{small_condition7}) that for any $R>0$,
$$\left\|(\partial_t u^k, \partial_t d^k)\right\|_{L^\frac32([0,\tau_0], W^{-1,\frac32}(B_R))}\le C(R)<+\infty.$$
This implies that
\begin{equation}\label{trace}
\displaystyle (u(t),\nabla d(t))\rightarrow (u_0,\nabla d_0)\ {\rm{strongly\ in}}\ L^3_{\rm{loc}}(\R^3) \ {\rm{as}}\ t\downarrow 0.
\end{equation}
In particular, we have that $(u_0,\nabla d_0)\in C^0_*([0,\tau_0], L^3_{\rm{uloc}}(\R^3))$.

\medskip
\noindent{\bf Step 5}. {\it Characterization of the maximal time interval $T_0$}. Let $T_0>\tau_0$ be the maximal time interval in which
the solution $(u,d)$ constructed in step 4 exists.  Suppose that $T_0<+\infty$ and (\ref{blowup_criterion}) were false. Then there exists
$r_0>0$ so that
$$\limsup_{t \uparrow T_0} |||(u(t), \nabla d(t)) |||_{L_{r_0}^3(\R^3)}\le \epsilon_0.$$
In particular, there exists $r_1\in (0, r_0]$ such that
$$\sup_{T_0-r_1^2\le t\le T_0} |||(u(t), \nabla d(t)) |||_{L_{r_1}^3(\R^3)}\le \epsilon_0.$$
Hence by Theorem \ref{smooth:solutions}, we conclude that $(u,d)\in C^\infty_b(\R^3\times [0, T_0])
\cap L^\infty([0, T_0], L^3_{\rm{uloc}}(\R^3))$. This contradicts the maximality of $T_0$. Hence (\ref{blowup_criterion}) holds.

\medskip
\noindent{\bf Step 6}. {\it Uniqueness}.  Let $(u_1, d_1), (u_2,d_0):\R^3\times [0, T_0]\to\R^3\times S^2$ be two solutions
of (\ref{LLF}), under the same initial condition $(u_0,d_0)$, that satisfy the properties of Theorem \ref{wellposedness}.
We first show $(u_1,d_1)\equiv (u_2,d_2)$ in $\R^3\times [0,\tau_0]$. This can be done
by the argument of \cite{wang} page 15-16. For convenience, we sketch it here.

Set $u=u_1-u_2, d=d_1-d_2$.  Then $(u,d)$ satisfies
\begin{eqnarray*}\begin{cases}
&\partial_t u-\Delta u = -\mathbb P\nabla\cdot[u_1\otimes u_1-u_2\otimes u_2 +\nabla d_1\odot\nabla d_1-\nabla d_2\odot\nabla d_2]\\
&\partial_t d-\Delta d = -(u_1\cdot\nabla d_1-u_2\cdot\nabla d_2)+|\nabla d_1|^2d_1-|\nabla d_2|^2 d_2\\
& (u,d)\ |_{t=0}= (0,0).
\end{cases}
\end{eqnarray*}
By the Duhamel formula, we have
\begin{eqnarray*}
\begin{cases}
& u(t)=-\mathbb V[u_1\otimes u_1-u_2\otimes u_2 +\nabla d_1\odot\nabla d_1-\nabla d_2\odot\nabla d_2]\\
& d(t)=-\mathbb S[(u_1\cdot\nabla d_1-u_2\cdot\nabla d_2)-(|\nabla d_1|^2d_1-|\nabla d_2|^2 d_2)],
\end{cases}
\end{eqnarray*}
where
$$\mathbb Sf(t)=\int_0^t e^{-(t-s)\Delta} f(s)\,ds,
\ \mathbb Vf(t)=\int_0^t e^{-(t-s)\Delta}\mathbb P\nabla\cdot f(s)\,ds, \ \forall f:\mathbb R^3\times [0,+\infty)\to\mathbb R^3.$$
Recall the three function spaces used in \cite{wang}.
Let $\mathbf X_{\tau_0}$ denote the space of functions $f$ on $\R^3\times [0,\tau_0]$ such that
$$|||f|||_{\mathbf X_{\tau_0}}:=\sup_{0<t\le\tau_0}\|f(t)\|_{L^\infty(\R^3)}+\|f\|_{X_{\tau_0}}<+\infty,$$
where
$$\|f\|_{\mathbf X_{\tau_0}}:=\sup_{0<t\le\tau_0}\sqrt{t}\|\nabla f(t)\|_{L^\infty(\R^3)}
+\sup_{x\in\R^3, 0<r\le\sqrt{\tau_0}}(r^{-3}\int_{P_r(x, r^2)}|\nabla f|^2)^\frac12,$$
 $\mathbf Y_{\tau_0}$ denote the space of functions $g$ on $\R^3\times [0,\tau_0]$ such that
$$||g||_{\mathbf Y_{\tau_0}}:=\sup_{0<t\le\tau_0}t\|g(t)\|_{L^\infty(\R^3)}
+\sup_{x\in\R^3, 0<r\le\sqrt{\tau_0}}r^{-3}\int_{P_r(x, r^2)}|g|<+\infty,$$
and $\mathbf Z_{\tau_0}$ the space of functions $h$ on $\R^3\times [0,\tau_0]$ such that
$$\|h\|_{\mathbf Z_{\tau_0}}:=\sup_{0<t\le\tau_0}\sqrt{t}\|h(t)\|_{L^\infty(\R^3)}
+\sup_{x\in\R^3, 0<r\le\sqrt{\tau_0}}(r^{-3}\int_{P_r(x, r^2)}|h|^2)^\frac12<+\infty.$$
Since $(u_i,d_i)\in L^\infty([0,\tau_0], L^2(\R^3)\times \dot{W}^{1,2}(\R^3))$ satisfies (\ref{small_condition0}) for $i=1,2$,
Theorem \ref{smooth:solutions} and the H\"older inequality imply that $u_i\in \mathbf Z_{\tau_0}, d_i\in \mathbf X_{\tau_0}$
for $i=1,2$, and
$$\sum_{i=1}^2 (\|u_i\|_{\mathbf Z_{\tau_0}}+\|d_i\|_{\mathbf X_{\tau_0}})\le C\epsilon_0.$$
It follows from Lemma 3.1 and Lemma 4.1 of \cite{wang} that
\begin{eqnarray*}
\|u\|_{\mathbf Z_{\tau_0}}+|||d|||_{\mathbf X_{\tau_0}}
&\lesssim& \Big\|(|u_1|+|u_2|)|u|+(|\nabla d_1|+|\nabla d_2|)|\nabla d|\Big\|_{\mathbf Y_{\tau_0}}\\
&&+\Big\| |u||\nabla d_2|+|u_1||\nabla d|+(|\nabla d_1|+|\nabla d_2|)|\nabla d|+|\nabla d_2|^2|d|\Big\|_{\mathbf Y_{\tau_0}}\\
&\lesssim& [\sum_{i=1}^2(\|d_i\|_{\mathbf X_{\tau_0}}+\|u_i\|_{\mathbf Z_{\tau_0}})]\|u\|_{\mathbf Z_{\tau_0}}
+[\sum_{i=1}^2(\|u_i\|_{\mathbf Z_{\tau_0}}+\|d_i\|_{\mathbf X_{\tau_0}})]|||d|||_{\mathbf X_{\tau_0}}\\
&\lesssim& \epsilon_0 [\|u\|_{\mathbf Z_{\tau_0}}+|||d|||_{\mathbf X_{\tau_0}}].
\end{eqnarray*}
This clearly implies that $(u_1,d_1)\equiv (u_2,d_2)$ in $\R^3\times [0,\tau_0]$. Since $(u_1, d_1)$ and $(u_2,d_2)$ are classical solutions of (\ref{LLF}) in
$\R^3\times [\tau_0, T_0)$, and $(u_1,d_1)=(u_2,d_2)$ at $t=\tau_0$, it is well-known that $(u_1,d_1)\equiv  (u_2,d_2)$ in
$\R^3\times [\tau_0, T_0)$. The proof is complete. \qed\\

Finally, we provide the proof of Lemma \ref{approx}. 

\smallskip
\noindent{\bf Proof of Lemma \ref{approx}}: 
Let $\theta\in C^\infty([0,+\infty))$ be such that
$$\theta(r)=1 \ {\rm{for}}\ 0\le r\le 1; \ \
0\le\theta(r)\le 1 \ {\rm{for}}\ 1\le r\le 2;\ \
\theta(r)=0 \ {\rm{for}}\ r\ge 2.$$ 
Let $\eta\in C_0^\infty(\R^3)$ be a standard mollifier,
and define for $k\ge 1$
$$\displaystyle\eta_{\frac1k}(x)={k^3}\eta(kx) \ {\rm{ and }}\
\theta_k(x)=\theta(\frac{x}{k}) \ {\rm{for}} \ x\in\R^3.$$
\noindent{\it Step 1}. {\it Approximation of $d_0$}. This will be done by two rounds of approximation.
It follows from $(d_0-e_0)\in L^3(\R^3)$  that
there exists $k_0>1$ such that for any $k\ge k_0$, it holds
\begin{equation}\label{fubini1}
\int_{\R^3\setminus B_{k-1}}|d_0-e_0|^3\le\epsilon_0^3.
\end{equation}
By the Fubini theorem, we may assume that for $k\ge k_0$, it also holds
\begin{equation}\label{fubini2}
\begin{cases}
\int_{\partial B_{k}}\left|d_0-e_0\right|^3\,dH^2
\leq 2\int_{\R^3\setminus B_{k-1}}\left|d_0-e_0\right|^3\le2\epsilon_0^3,\\
\sup_{x\in\partial B_k}\int_{\partial B_k\cap B_2(x)}|\nabla d_0|^3\,dH^2
\le 4\left\|\nabla d_0\right\|_{L^3_2(\R^3)}^3\le C\epsilon_0^3,\\
\int_{\partial B_k}\left|\nabla d_0\right|^3\,dH^2
\le 2\int_{B_{k+1}}\left|\nabla d_0\right|^3\lesssim k^3
\left\|\nabla d_0\right\|_{L^3_1(\R^3)}^3
\le k^3\epsilon_0^3.
\end{cases}
\end{equation}
Define the approximate sequence $\widetilde{d_0^k}:\mathbb R^3\to\mathbb R^3$
by 
\begin{equation*}
\widetilde{d_0^k}(x)= \begin{cases}d_0(x) & \ {\rm{if}}\ |x|\le k\\
(|x|-k)e_0+(k+1-|x|) d_0(k\frac{x}{|x|}) & \ {\rm{if}}\ k\le |x|\le k+1\\
e_0 & \ {\rm{if}}\ |x|\ge k+1.
\end{cases}
\end{equation*}
Then by direct calculations we have that
\begin{eqnarray*}
\|\nabla\widetilde{d_0^k}\|_{L^(\R^3)}^p
&=& \int_{B_k}|\nabla d_0|^p+\int_{B_{k+1}\setminus B_k}|\nabla\widetilde{d_0^k}|^p\\
&\lesssim& \int_{B_k}|\nabla d_0|^p
+\int_{\partial B_k}|\nabla d_0|^p\,dH^2+\int_{\partial B_k}|d_0-e_0|^p\,dH^2\\
&\lesssim& k^p\epsilon_0^p<+\infty, \  {\rm{for}}\ p=2, 3,
\end{eqnarray*}
\begin{eqnarray*}
\left\|\nabla \widetilde{d_0^k}\right\|_{L^3_1(\R^3)}^3
&\lesssim &
\left\|\nabla d_0\right\|_{L^3_1(\R^3)}^3
+\sup_{x\in\partial B_k}\int_{\partial B_k\cap B_1(x)}|d_0-e_0|^3\,dH^2\\
&&+\sup_{x\in\partial B_k}\int_{\partial B_k\cap B_1(x)}|\nabla d_0|^3\,dH^2\\
&\leq& C\epsilon_0^3,
\end{eqnarray*}
and for any $x_0\in B_{k+1}\setminus B_k$,
\begin{eqnarray*}
  {\rm{dist}}(\widetilde{d_0^k}(x_0),S^2) &\leq& \frac{1}{|B_1|}\int_{B_1(x_0)}
\left |\widetilde{d_0^k}(x_0) - d_0(y)\right|\\
&\lesssim& \int_{B_1(x_0)} \left|(|x_0|-k)e_0+(k+1-|x_0|)d_0(k\frac{x_0}{|x_0|})-d_0(y)\right|\\
&\lesssim& \int_{B_1(x_0)} \left|d_0(y)-e_0\right|+\left|d_0(y)-d_0(k\frac{x_0}{|x_0|})\right|\\
&\lesssim& \left(\int_{\R^3\setminus B_k}|d_0-e_0|^3\right)^\frac13
+\left\|\nabla d_0\right\|_{L^3_1(\R^3)}
\le 2\epsilon_0.
\end{eqnarray*}
This implies
$$\sup_{x_0\in\R^3}  {\rm{dist}}(\widetilde{d_0^k}(x_0),S^2) 
=\sup_{x_0\in B_{k+1}\setminus B_k}  {\rm{dist}}(\widetilde{d_0^k}(x_0),S^2) 
\le 2\epsilon_0$$
so that  if  $\epsilon_0>0$ is chosen sufficiently small then
$\widetilde{d_0^k}(x)$ remains close to $S^2$ uniformly for $x\in\R^3$.
Therefore we can project $\widetilde{d_0^k}$ onto $S^2$ to get
$\displaystyle \widehat{d_0^k}(x)=\frac{\widetilde{d_0^k}(x)}{|\widetilde{d_0^k}(x)|}$
for $x\in \R^3$. It is easy to see that $\displaystyle \widehat{d_0^k}:\R^3\to S^2$ satisfies:
\begin{equation}\label{approx1}
\widehat{d_0^k}=d_0 \ {\rm{in}}\ B_k, \ \widehat{d_0^k}=e_0 \ {\rm{in}}\ \R^3\setminus B_{k+1},\ 
 \left\|\nabla\widehat{d_0^k}\right\|_{L^3_1(\R^3)}\le C\epsilon_0, 
\ {\rm{and}}\ \ \int_{\R^3}\left|\nabla \widehat{d_0^k}\right|^p\le Ck^p\epsilon_0^p<+\infty\ (p =2, 3).
\end{equation}
For any $l, k\ge 1$, define $\displaystyle d_0^{k, l}(x)=\left(\eta_{\frac1l}*\widehat{d_0^k}\right)(x)$ for $x\in\R^3$. 
Then $d_0^{k,l}\in C^\infty(\R^3,\R^3)$ satisfies
\begin{equation}\label{approx2}
\left\|\nabla {d_0^{k,l}}\right\|_{L^3_1(\R^3)}\le C\epsilon_0, 
\ \ {\rm{and}}\ \ \int_{\R^3}\left|\nabla {d_0^{k,l}}\right|^p\le Ck^p\epsilon_0^3<+\infty, \ \forall\ l\ge 1, 
\ (p=2, 3),
\end{equation}
and by the modified Poincar\'e inequality it holds that
\begin{equation}\label{small_distance}
\sup_{x\in\R^3}{\rm{dist}}(d^{k,l}_0(x), S^2)\lesssim \left\|\nabla d_0^{k,l}\right\|_{L^3_1(\R^3)}
\le C\epsilon_0, \ \forall \ l\ge 1,
\end{equation}
and for any $k\ge 1$, 
$$\lim_{l\rightarrow\infty} 
\left(\|d^{k,l}_0-d_0\|_{L^p(B_{k-1})}+\|\nabla (d^{k,l}_0-d_0)\|_{L^(B_{k-1})}\right)=0, \ {\rm{for}}\ p=2, 3.$$
Therefore, by the Cauchy diagonal process we may conclude that, after taking possible subsequences, 
there exist $l(k)\rightarrow\infty$ as $k\rightarrow\infty$ such that
$$d_0^{k}(x)=\frac{d_0^{k,l(k)}}{\left|d_0^{k, l(k)}\right|}(x), \ \forall\ x\in\R^3,$$
satisfies the desired properties of approximation: $d_0^k\in C^\infty(\R^3, S^2)\cap \dot{W}^{1,p}(\R^3,S^2)$ ($p=2,3$), and
\begin{equation}\label{approx5}
\left\|\nabla d_0^k\right\|_{L^3_1(\R^3)}\le C_0\epsilon_0,
\end{equation}
and for any $0<R<+\infty$, 
\begin{equation}\label{approx6}
\lim_{k\rightarrow\infty}\left [\|d_0^k-d_0\|_{L^p(B_R)}+\|\nabla(d_0^k-d_0)\|_{L^p(B_R)}\right]=0, \ {\rm{for}}\ p=2,3.
\end{equation}

\medskip
Next we would like to obtain the desired approximation of $u_0$, whose proof is similar to \cite{basson} Theorem 1.4. For the completeness,
we outline the detail below.
 
\smallskip
\noindent{\it Step 2}. {\it Approximation of $u_0$}.  
Let $\mathbb P:L^2(\R^3)\to \mathbb PL^2(\R^3)$ denote the Leray projection operator. 
For $k\ge 1$, define 
\begin{equation*}
  \label{approx:seqs}
   \widetilde{ u_0^k}(x)=\mathbb P[\theta_k u_0](x), \ x\in\R^3.
\end{equation*}
Since $\theta_k u_0\in L^p(\R^3,\R^3)$ and $\mathbb P:L^p(\R^3)\to \mathbb PL^p(\R^3)$ 
is bounded, it follows that $\nabla\cdot \widetilde{u_0^k}=0$ in $\mathbb R^3$
and $\widetilde{u_0^k}\in L^p(\R^3)$ for $p=2,3$.  
Now we want to show 
\begin{equation}\label{l3_loc_bound}
\left\|\widetilde{u_0^k}\right\|_{L^3_{1}(\R^3)}\lesssim \Big\|u_0\Big\|_{L^3_{1}(\R^3)},
\end{equation}
and 
\begin{equation}\label{l3_loc_conv}
\widetilde{u_0^k}\rightarrow  u_0
\ {\rm{strongly\ in}}\ L^p_{\rm{loc}}(\R^3)\  {\rm{for}}\  p=2,3. 
\end{equation}
Since 
$$\widetilde{ u_0^k}(x)=(\theta_k u_0)(x)-\nabla\Delta^{-1}\nabla\cdot[\theta_k u_0](x),$$
and $\displaystyle\|\theta_k u_0\|_{L^3_1(\R^3)}\le \|u_0\|_{L^3_1(\R^3)}$, it suffices to show
$$\Big\|\nabla\Delta^{-1}\nabla\cdot[\theta_k u_0]\Big\|_{L^3_1(\R^3)}
\lesssim  \Big\|u_0\Big\|_{L^3_{1}(\R^3)}.$$
Set $\Phi=\nabla\Delta^{-1}\nabla\cdot[\theta_k u_0]$. Then we have  
$$\widetilde{u_0^k}(x)=\theta(\frac{x}{k})u_0(x)-\Phi(x), \  x\in\R^3.$$
It follows from  $\nabla\cdot u_0=0$ that we have
\begin{eqnarray*}
\Phi(x)&=&\nabla\Delta^{-1}\nabla\cdot[\theta_k u_0](x)=\nabla\Delta^{-1}[(\nabla \theta_k)\cdot u_0](x)\\
&=& \frac1{k}\int_{\R^3}K(x-y)\nabla\theta(\frac{y}{k})\cdot u_0(y)\\
&=&\frac1{k}\int_{B_k(x)}K(x-y)\nabla\theta(\frac{y}{k})\cdot u_0(y)
+\frac1{k}\int_{\R^3\setminus B_k(x)}K(x-y)\nabla\theta(\frac{y}{k})\cdot u_0(y)\\
&=&I(x)+II(x),
\end{eqnarray*}
where $\displaystyle K(x)=c_3\frac{x}{|x|^3}, \ c_3=\frac{1}{3|B_1|}$,  is the kernel of the operator $\nabla\Delta^{-1}$. We estimate $I$ and $II$ separately as follows.
It is easy to see that
$$\left\|I\right\|_{L^3_1(\R^3)}
\le \frac{1}{k}\left\|K\right\|_{L^1(B_{k})}\left\|\nabla\theta(\frac{\cdot}{k}) \cdot u_0\right\|_{L^3_1(\R^3)}
\le C\Big\|u_0\Big\|_{L^3_1(\R^3)},$$
while
$$\left|II(x)\right|\le \frac{C}{k^3}\int_{B_{2k}}|u_0(y)|\le C\Big\|u_0\Big\|_{L^3_1(\R^3)},$$
so that
$$\left\|II\right\|_{L^3_1(\R^3)}\le C \left\|u_0\right\|_{L^3_1(\R^3)}.$$
Combining these two estimates implies (\ref{l3_loc_bound}). 

For any fixed compact set $E\subset\R^3$ and $x\in E$, we write
\begin{eqnarray*}
\Phi(x)&=&\frac{c_3}{k}\int_{\R^3}\left(\frac{x-y}{|x-y|^3}+\frac{y}{|y|^3}\right)\nabla\theta(\frac{y}{k})\cdot u_0(y)
-\frac{c_3}{k}\int_{\R^3}\frac{y}{|y|^3}\nabla\theta(\frac{y}{k})\cdot u_0(y)\\
&=&III_k(x)+IV_k(u_0).
\end{eqnarray*}
Since $\nabla\theta(\frac{y}{k})$ has its support in $B_{2k}\setminus B_k$,  for $k$ sufficiently large we have that
$$\displaystyle\left |\frac{x-y}{|x-y|^3}+\frac{y}{|y|^3}\right|\le \frac{C_E}{k^3}, 
\ {\rm{ for }}\ x\in E \ {\rm{and}}\ y\in B_{2k}\setminus B_k,$$ 
and hence  it holds
$$
|III_k(x)|\leq \frac{C_E}{k^4}\int_{B_{2k}\setminus B_k}|u_0(y)|\leq \frac{C_E}{k}\Big\|u_0\Big\|_{L^3_1(\R^3)}\rightarrow 0
\ {\rm{as}}\ k\rightarrow\infty,
$$
while it is easy to bound $IV_k(u_0)$  by
$$|IV_k(u_0)|\lesssim \frac{1}{k^3}\int_{B_{2k}}|u_0(y)|\lesssim \Big\|u_0\Big\|_{L^3_1(\R^3)}.$$
Hence we may assume that there exists a constant vector $c\in \R^3$, with $\displaystyle |c|\le C\|u_0\|_{L^3_1(\R^3)}$, such that
$$\lim_{k\rightarrow\infty} IV_k(u_0)=c.$$
Now we define
$$\widehat{u_0^k}(x)=\widetilde{u_0^k}(x)+\frac32 \mathbb P[\theta(\frac{x}{k})c], \ x\in\R^3.$$
Then we have that $\widehat{u_0^k}\in L^p(\mathbb R^3)$ for $p=2,3$, and
$$\Big\|\widehat{u_0^k}\Big\|_{L^3_1(\R^3)}\le C\Big\|u_0\Big\|_{L^3_1(\R^3)}.$$
It is easy to check that for any $x\in E$, if $k\rightarrow\infty$ then 
$$\mathbb P[\theta(\frac{x}{k})c]=\theta(\frac{x}{k})c-\nabla\Delta^{-1}\nabla\cdot[\theta(\frac{x}{k}) c]
=\theta(\frac{x}{k})c+o(1)+\frac{c_3}{k}\int_{\R^3}\frac{y}{|y|^3}\nabla\theta(\frac{y}{k})\cdot c
\rightarrow \frac23 c.$$
Therefore, for any $x\in E$, if $k\rightarrow\infty$ then 
\begin{eqnarray*}
\widehat{u_0^k}(x)-u_0(x)&=&(\theta(\frac{x}{k})-1)u_0(x)-\Phi(x)+\frac32 \mathbb P[\theta(\frac{x}{k})c]\\
&=&(\theta(\frac{x}{k})-1)u_0(x)-III_k(x)-IV_k(u_0)+\frac32 \mathbb P[\theta(\frac{x}{k})c]\rightarrow 0.
\end{eqnarray*}
This clearly implies (\ref{l3_loc_conv}). The proof of Lemma \ref{approx} is not complete yet,
since $\widehat{u_0^k}\notin C^\infty(\R^3,\R^3)$. To overcome this, we mollify $\widehat{u_0^k}$ to get
$$u_0^{k,l}(x)=\left(\eta_{\frac{1}{l}}* u_k \right)(x), \ x\in\R^3, \forall \ l\ge 1.$$
Then it is straightforward to check that $u_0^{k,l}\in C^\infty(\R^3,\R^3)\cap L^p(\R^3,\R^3)$ for $p=2,3$,
$\nabla\cdot u_0^{k,l}=0$, 
$$\Big\|u_0^{k,l}\Big\|_{L^3_1(\R^3)}\le \Big\|\widehat{u_0^k}\Big\|_{L^3_1(\R^3)}\le C\Big\|u_0\Big\|_{L^3_1(\R^3)},$$
and for any $k\ge 1$,
$$u_0^{k,l}\rightarrow \widehat{u_0^k} \ {\rm{strongly\ in}}\ L^p_{\rm{loc}}(\R^3) \ {\rm{for}}\ p=2,3,
\ {\rm{as}}\ l\rightarrow \infty.$$
Thus, by the Cauchy diagonal process we may assume that there exist $l(k)\rightarrow\infty$ as $k\rightarrow\infty$ such that
$$u_0^k(x)=u_0^{k,l(k)}(x), \ x\in\R^3$$ satisfies the required properties of approximation of $u_0$:
$u_0^k\in C^\infty(\R^3,\R^3)\cap L^p(\R^3,\R^3)$ for $p=2,3$, $\nabla\cdot u_0^{k}=0$,
$$\Big\|u_0^{k}\Big\|_{L^3_1(\R^3)}\le C\Big\|u_0\Big\|_{L^3_1(\R^3)},$$
and
$$u_0^{k}\rightarrow {u_0} \ {\rm{strongly\ in}}\ L^p_{\rm{loc}}(\R^3) \ {\rm{for}}\ p=2,3,
\ {\rm{as}}\ k\rightarrow \infty.$$
This completes the proof of Lemma \ref{approx}.     \qed

\bigskip

\noindent{\bf Acknowledgements}. Both authors are partially supported by NSF grant 1000115. The second
author is also partially supported by NSFC grant 11128102. The paper is based on
the first author's PhD thesis \cite{hineman} at the University of Kentucky. The first author would like to thank
the Department of Mathematics for its support and Professor John Lewis for his interest in this work.

\end{document}